\documentclass[11pt,oneside,a4paper]{amsart}\openup1.6\jot
\usepackage{amsthm}
\usepackage{amsmath}
\usepackage{amsrefs} 
\usepackage{amssymb}
\usepackage{amsfonts}
\usepackage{amscd}
\usepackage{mathrsfs} 
\usepackage{mathtools}
\usepackage[dvipdfmx]{graphicx}
\usepackage{bm}   
\usepackage[all]{xy}
\usepackage{ulem}
\usepackage[all]{xy}
\usepackage{ulem}
\newtheorem{theorem}{Theorem}[section]
\newtheorem{lemma}[theorem]{Lemma}
\newtheorem{proposition}[theorem]{Proposition}
\newtheorem{corollary}[theorem]{Corollary}

\theoremstyle{definition}
\newtheorem{definition}[theorem]{Definition}
\newtheorem{example}[theorem]{Example}
\newtheorem{remark}[theorem]{Remark} 

\def\C{\mathbb{C}} 
\def\R{\mathbb{R}}
\def\Q{\mathbb{Q}}
\def\P{\mathbb{P}}
\def\Z{\mathbb{Z}}

\def\i{{\tt{i}}}

\def\id{\mathrm{id}}

\newcommand{\End}{\mathrm{End}}
\newcommand{\Hom}{\mathrm{Hom}}

\newcommand{\homscr}{{\mathscr{H}\! \! om}}

\def\({\left(}
\def\){\right)}
\def\<{\langle}
\def\>{\rangle}
\newcommand{\simeqto}{\xrightarrow{\sim}}
\newcommand{\jump}[1]{\ensuremath{[\![#1]\!]} }
\newcommand{\pole}[1]{\ensuremath{(\!(#1)\!)} }
\newcommand{\conv}[1]{\ensuremath{(\!\{#1\}\!)}}
\newcommand{\cal}[1]{\mathcal{#1}}
\renewcommand{\scr}[1]{\mathscr{#1}}
\renewcommand{\frak}[1]{\mathfrak{#1}}

\def\O{{\mathscr{O}}}
\def\per{{\mathrm{per}}}

\def\asy{{\mathrm{asy}}}
\def\gr{{\mathrm{gr}}}

\def\RH{{\mathrm{RH}}}
\def\DR{{\mathrm{DR}}}

\def\Diffc{{\mathsf{Diffc}}}

\def\aut{{\scr{A}ut}}

\title{Stokes structure of wild difference modules}
\author{Yota Shamoto}
\date{}

\begin{document}
\begin{abstract}
We formulate and prove a 
Riemann--Hilbert correspondence
between two categories: 
wild difference modules 
and wild Stokes-filtered 
$\mathscr{A}_{\rm{per}}$-modules. 
This correspondence is motivated by 
the Riemann--Hilbert correspondence 
for germs of meromorphic connections in one variable
due to Deligne--Malgrange. 
It also generalizes the 
Riemann--Hilbert 
correspondence for mild difference modules.  
\end{abstract}
\maketitle

\section{Introduction}
In this paper, 
we establish a 
Riemann--Hilbert correspondence 
between two categories. 
The first is the category of additive analytic difference modules; the
second is the category of wild Stokes-filtered $\scr{A}_{\rm{per}}$-modules.

This correspondence is 
a generalization 
of the Riemann--Hilbert correspondence 
\cite{shamoto2022stokes}
for mild difference modules,
and can be regarded as a difference analogue of the Deligne--Malgrange Riemann–Hilbert correspondence \cites{Deligne,Malgrange} for germs of meromorphic connections in one variable.
This generalization covers a new class of modules 
called wild difference modules,
as introduced in \cite{Galois}.

In this introduction,
we describe the two categories 
in \S \ref{WDM} and \S \ref{WSS}.
We then explain the Riemann--Hilbert functor 
between these categories in \S \ref{IRH}. 
\subsection{Wild difference modules}
\label{WDM}
Let 
$\C\conv{s^{-1}}$ be
the field of convergent Laurent series
in $s^{-1}$ with complex coefficients. 
We have an automorphism
$\phi\colon \C\conv{s^{-1}}\to \C\conv{s^{-1}}$
defined by $\phi (f)(s)=f(s+1)$.
A difference module is a pair $(\scr{M}, \psi)$ consisting of a $\C\conv{s^{-1}}$-vector space $\scr{M}$ and a $\C$-linear automorphism $\psi\colon \scr{M}\to \scr{M}$
which satisfies 
$\psi(fv)=\phi(f)\psi(v)$
for any $f\in \C\conv{s^{-1}}$
and $v\in\scr{M}$.
A difference module is 
called \textit{mild} if it 
is isomorphic to a 
module of the form
$(\C\conv{s^{-1}}^r,A(s)\phi)$
where $A(s)$ has holomorphic entries
in $s^{-1}$ whose constant term $A(\infty)$
is invertible. 

A difference module which is not mild is called \textit{wild} (cf. \cite{Galois}).
A typical example of
a wild difference module is 
$(\scr{M},\psi)=(\C\conv{s^{-1}},
s^{-1}\phi)$. 
A ``solution'' of this module 
is the Gamma function $\Gamma(s)$,
which admits the following asymptotic expansion as $s\to\infty$ for any $\varepsilon>0$, in the sector $-\pi+\varepsilon<\arg s<\pi-\varepsilon$:
  \[\Gamma(s)\sim \sqrt{2\pi}e^{-s}s^{s-1/2}\left(1+\frac{1}{12}s^{-1}+\cdots\right).\]
  
The essential new feature 
in the wild case is the appearance of exponential factors of the form $s^{\lambda s}$ with $\lambda\in \Q$, which have no analogue in the theory of meromorphic connections.

\subsection{Wild Stokes structure}\label{WSS}
The corresponding category 
of Stokes structures 
is a straightforward 
generalization of the mild case \cite{shamoto2022stokes}. 
To formulate the Stokes structure of mild difference modules,
we considered a sheaf $\scr{A}_\per$
of filtered rings on a circle $S^1$, 
regarded as
the real blow-up at $s=\infty$.
The sheaf $\scr{A}_\per$ is a sheaf of certain periodic functions with respect to $s\mapsto s+1$, given by Laurent power series in the variable $u=\exp(2\pi i s)$ or $u^{-1}$, depending on the sector.
We may define a filtration on $\scr{A}_{\per}$ by the growth order
(see \S \ref{Wild Stokes}).

A Stokes-filtered $\scr{A}_{\per}$-module
in the mild case \cite{shamoto2022stokes}
was a pair consisting of a locally free $\scr{A}_\per$-module $\scr{L}$
together with a 
family of subsheaves
\[\{\scr{L}_{\leqslant \frak{a}}\subset\scr{L}_{|U}
\mid U\subset S^1,\frak{a}\in \scr{I}^{\rm mild}(U)\}\]
indexed by a sheaf 
$\scr{I}^{\rm mild}=\sum_{r\in\Q\cap (0,1]}\C_{S^1} s^{r}
\subset\widetilde{\scr{O}}$
with some conditions.  
Here, 
the symbol $\C_{S^1}$ denotes the constant sheaf and 
the symbol $\widetilde{\scr{O}}$ denotes the sheaf of holomorphic functions on $\C^*$, pushed forward to the real blow-up 
space and pulled back to the boundary circle $S^1$
(see \S \ref{Wild Stokes}).

In the wild case, 
we also consider a
locally free $\scr{A}_\per$-module
with a filtration.
What is different from the mild case 
is the index sheaf $\scr{I}$
for the Stokes filtration. 
We use the following index sheaf instead of $\scr{I}^{\rm mild}$:
\begin{align*}
   \scr{I} =\Q_{S^1} s\log s+
\sum_{q\in \Q\cap(0,1]}\C_{S^1} s^{q}
\subset \widetilde{\scr{O}}.
\end{align*}
Except for this change of the index sheaf,
the theory of wild Stokes-filtered $\scr{A}_\per$-modules  
is completely parallel to that of
(mild) Stokes-filtered $\scr{A}_\per$-modules.

\subsection{Riemann--Hilbert functor}\label{IRH}
The main difference between 
the mild case and the wild case
is the construction of 
the Riemann--Hilbert functor. 
The Riemann--Hilbert functor defined in \cite{shamoto2022stokes}
for mild difference modules 
does not work well
for wild difference modules. 
The difficulty appears when 
we deal with the function $s^s$
in the construction of the functor.
Since $s^s$ grows or decays faster than any power of $u=\exp(2\pi i s)$, naive
definitions of the de Rham functor 
cannot capture $s^s$,
or else become too large to be manageable. 

To overcome the difficulty,
we introduce a new sheaf 
$\scr{Q}$ of rings on $S^1$
as a subsheaf of $\widetilde{\scr{O}}$. 
Roughly speaking, $\scr{Q}$ consists of holomorphic functions on sectors at infinity that satisfy a suitable moderate growth condition in the $u$-direction.
(see \S 
\ref{Riemann--Hilbert correspondence}).

We then define the de Rham functor 
using the sheaf $\scr{Q}$ instead of 
$\widetilde{\scr{O}}$. By considering the wild Stokes filtration on it, we obtain the Riemann--Hilbert functor $\RH$. 
The main result of the paper is as follows: 
\begin{theorem}[Theorem \ref{main theorem}]
Let $ \mathsf{Diffc}$
denote the category of difference modules. Let $\mathsf{St}^{\rm wild}(\mathscr{A}_\per)$
be the category of wild Stokes-filtered $\scr{A}_\per$-modules. 
    The functor 
    \[\RH\colon \mathsf{Diffc}
    \longrightarrow
    \mathsf{St}^{\rm wild}(\mathscr{A}_\per),
    \quad \mathscr{M}\mapsto 
    \RH(\scr{M})
    \]
    is an equivalence of categories.
\end{theorem}
The proof uses the classical theories 
of linear difference equations \cites{Immink,Immink2,MR1091837,Galois}. 
The main technical difficulty is to verify that these classical theories remain valid over $\scr Q$. We make this precise in \S \ref{technical}.  
The proofs of the statements formulated in \S \ref{technical} 
are provided in \S \ref{Proof Q}. 

The author hopes that this 
theory of Stokes structure for 
wild difference modules contributes to a better understanding of wild difference modules.
In particular, 
the relation to the algebraic Mellin transformation \cites{GS, lopez2023formal} and to difference Galois theory \cite{Galois} would be interesting to investigate.
Moreover, comparing this Riemann--Hilbert correspondence 
with that of $q$-difference equations
\cites{RamisSauloyZhang2013, vanDerPutReversat2007} deserves further study.

\subsection{Outline of the paper}
In \S \ref{Wild Stokes},
we define the notion of wild Stokes-filtered 
$\scr{A}_\per$-modules. 
In \S \ref{Riemann--Hilbert correspondence},
we formulate the Riemann--Hilbert correspondence
for wild difference modules. 
We also give a proof assuming some 
technical theorems.
The technical theorems assumed in \S \ref{Riemann--Hilbert correspondence} will be proved
in \S \ref{Proof Q}. 

\subsection{Acknowledgments}
The author would like to express his deep gratitude to Takuro Mochizuki and Claude Sabbah for their enlightening suggestions
and encouragement on many occasions. 
He also thanks Tatsuki Kuwagaki, Yosuke Ohyama, Takahiro Saito, and Fumihiko Sanda for discussions and encouragement.

This study is supported by JSPS KAKENHI Grant Numbers 24K16925 and 20K14280. 
This work was also supported by the Research Institute for Mathematical Sciences,
an International Joint Usage/Research Center located in Kyoto University.

\subsection{Notation}
We fix $\i=\sqrt{-1}$. 
For a complex manifold $M$, we denote by $\scr{O}_M$ the sheaf of holomorphic functions on $M$.
For a sheaf $\scr{F}$ on a topological space $X$, 
the restriction of $\scr{F}$ to an open subset $U\subset X$ is denoted by $\scr{F}_{|U}$, 
and the stalk at a point $x\in X$ is denoted by $\scr{F}_x$. 
The real and imaginary parts of a complex number $c$ are denoted by $\mathrm{Re}(c)$ and $\mathrm{Im}(c)$, respectively.

\section{Wild Stokes filtration}\label{Wild Stokes}
\subsection{Definition and fundamental properties}
\subsubsection{Preliminary}
Let 
$S^1=\{w\in \mathbb{C}
\colon |w|=1\}$ 
be the unit circle. 
We consider the real oriented blow-up 
$\widetilde{\C}\coloneqq 
\{(z,w)\in \mathbb{C}
\times S^1\mid z=|z|w \}$. 
Let
$\widetilde{\jmath}\colon 
\mathbb{C}^*=\C\setminus\{0\}
\to \widetilde{\C}$
and 
$\widetilde{\imath}\colon {S}^1\to \widetilde{\C}$
be the inclusions. 
The circle 
$S^1$ is regarded
as the boundary of $\widetilde{\C}$.
Let $\scr{O}_{\C^*}$
denote the sheaf of holomorphic functions 
on $\C^*$. 
We use the notation 
$\widetilde{\O}\coloneqq 
\widetilde{\imath}^{-1}\widetilde{\jmath}_*\O_{\C^*}$.
For later use, 
we set $s=z^{-1}$. 
The logarithmic function
$\log s$
is regarded as a 
local section of 
$\widetilde{\O}$
if we fix a branch. 
For two real numbers $a,b$ with $a<b$,
we set  
$(a,b)_s=\{a<\arg(s)<b\}
=\{e^{-\i\theta}\mid a<\theta<b\}\subset S^1$. 
We also set 
$S^1_{\rm R}=(-\pi/2,\pi/2)_s$
and $S^1_{\rm L}=(\pi/2,3\pi/2)_s$. 

\subsubsection{Index sheaf}\label{Index}
We shall introduce a
sheaf of ordered 
abelian groups. 
\begin{definition}
We define a subsheaf 
$\scr{I}\subset \widetilde{\scr{O}}$ as follows:
    \begin{align*}
    \mathscr{I}
    &=\{
rs\log s+\sum_{j=1}^p c_j s^{j/p}\mid p\in \Z_{>0}, 
r\in p^{-1}\Z, 
c_j\in\C \}\\
&=\Q_{S^1} s\log s+
\sum_{q\in \Q\cap(0,1]}\C_{S^1} s^{q}
\subset \widetilde{\scr{O}}. 
\end{align*}
Note that the choice of a branch of $\log s$
(and hence $s^{1/p}=\exp(p^{-1}\log s)$) does
not affect the definition of the sheaf $\scr{I}$. 
\end{definition}
We define a partial order on the sections of $\scr{I}$ as follows:
Let 
$\frak{a},\frak{b}$
be two sections
of $\scr{I}(V)$
for an open subset 
$V\subset S^1$. 
Then 
we set 
$\frak{a}\leqslant_V
\frak{b}$ (resp. $\frak{a}<_V \frak{b}$)
if 
$\exp(\frak{a}-\frak{b})$
is of moderate growth (resp. rapid decay;
see \S \ref{filtration} to recall these terms). 
For $e^{\i\theta}\in V$,
we set $\frak{a}\leqslant_\theta \frak{b}$
if there exists a neighborhood $U\subset V$ of $e^{\i\theta}$
such that $\frak{a}\leqslant_U\frak{b}$. 
We define $\frak{a}<_\theta\frak{b}$ similarly. 
The following lemma provides a more explicit description of this condition:
\begin{lemma}\label{leqslant}
Take a point $x\in S^1$ and set 
$U=S^1\setminus \{x\}$. Put 
$U_{\rm R}=U\cap S^1_{\rm R}$ and
$U_{\rm L}=U\cap S^1_{\rm L}$.
Fix a branch of $\log s$  on $U$.
Take a positive integer 
$p$, an integer $q\in\Z$, and complex numbers
$c_1,\dots, c_p$. 
Let us consider a section $\frak{a}\in \scr{I}(U)$ given by 
\begin{align*}
    \frak{a}=qs\log (s^{1/p})+\sum_{j=1}^p c_js^{j/p} 
\end{align*}
where  $s^{1/p}=\exp(p^{-1}\log s)$.
\begin{enumerate}
    \item If $q>0$, then 
    we have $\frak{a}
    \leqslant_{U_{\rm L}}0$ and 
    $0\leqslant_{U_{\rm R}} 
    \frak{a}$.
    \item If $q<0$, then we have
    $0\leqslant_{U_{\rm L}}\frak{a}$
    and $\frak{a}
    \leqslant_{U_{\rm R}} 
    0$. 
    \item If $q=0$ and $\frak{a}\neq 0$, let $j_0=\max\{j\mid c_j\neq 0\}$,
    $U_{+}\coloneqq U\cap
    \{\mathrm{Re}
    (e^{\i j_0 \theta/p}c_{j_0})>0\}$,
    and $U_-\coloneqq U\cap 
    \{\mathrm{Re}
    (e^{\i j_0 \theta/p}c_{j_0})<0\}$
    where 
    $\theta\coloneqq
    \mathrm{Im}(\log s)$. 
    Then we have 
    $0\leqslant_{U_+} \frak{a}$ 
    and $\frak{a}\leqslant_{U_-}0 $.
\end{enumerate}
\end{lemma}
\begin{proof}
Since $|e^x|=e^{\mathrm{Re}(x)}$, 
we consider the functions
$\mathrm{Re}(\frak{a}(s))$ as 
$|s|\to \infty$
with $s^{-1}|s|\in U$.
Then if $q\neq 0$,
the dominant term is
\[\mathrm{Re}(qs\log s^{1/p})
=p^{-1}q(\mathrm{Re}(s)\log |s|-
\mathrm{Im}(s)\theta).
\] Since $\theta$ is bounded on each compact subset of $U$, we obtain (1) and (2). 
When $q=0$ and $\frak{a}\neq 0$,
the leading term is 
$\mathrm{Re}(c_{j_{0}}s^{j_0/p})$.
Hence we obtain $(3)$. 
\end{proof}

\subsubsection{Wild Stokes-filtered 
$\scr{A}_{\per}$-modules}
Let 
$\mathbb{C}\{u\}$
denote the field of 
convergent power 
series in $u$. 
Let
$\mathbb{C}\conv{u}=\mathbb{C}\{u\}[u^{-1}]$
be
the quotient field. 
For a connected 
open subset 
$V\subset S^1$, 
we set
\begin{align*}
    \scr{A}_{\per}(V)
    =\begin{cases}
    \mathbb{C}\conv{u}&(V\subset (0,\pi)_s)\\
    \mathbb{C}\conv{u^{-1}}&(V\subset (-\pi,0)_s)\\
    \mathbb{C}
    [u,u^{-1}]&(\text{otherwise} ).
    \end{cases}
\end{align*}
This defines 
the sheaf 
on $S^1$
denoted by 
$\scr{A}_{\per}$. 
We define 
a sheaf $\scr{A}_{\per}^{\leqslant 0}$ of subrings in
$\scr{A}_{\per}$
as follows: 
\begin{align*}
    \scr{A}_{\per}^{\leqslant 0}(V)
    =\begin{cases}
    \mathbb{C}\{u\}&(V\subset (0,\pi)_s)\\
    \mathbb{C}\{u^{-1}\}&(V\subset (-\pi,0)_s)\\
    \mathbb{C}&(\text{otherwise}),
    \end{cases}
\end{align*}
where $V\subset S^1$
denotes a connected open subset. 
The sheaf of maximal ideals of 
$\mathscr{A}_{\per}^{\leqslant 0}$
is denoted by 
$\scr{A}^{<0}_{\per}$. 
 
\begin{definition}\label{preStokes}
    For a $\scr{A}_{\per}$-module
    $\scr{L}$, 
    a \textit{wild pre-Stokes filtration}
    is a family 
    \[\scr{L}_{\leqslant\bullet}=
    \{\scr{L}_{\leqslant \frak{a}} 
    \subset \scr{L}_{|U} \mid U\subset S^1 
    \text{ is open, } 
    \frak{a}\in \scr{I}(U)\}\] 
    of 
    $\scr{A}_{\per |U}^{\leqslant 0}$-submodules
    satisfying the following conditions: 
    \begin{enumerate}
        \item If $\mathfrak{b}_{|U}=\frak{a}$
        for $U\subset V$, $\frak{a}\in\scr{I}(U)$
        and $\frak{b}\in \scr{I}(V)$, then 
        $\mathscr{L}_{\leqslant \frak{b}|U}
        =\mathscr{L}_{\leqslant \frak{a}}$.
        \item
        If 
        $\mathfrak{a}\leqslant_V \frak{b}
        $ for $\frak{a},\frak{b}\in \scr{I}(V)$,
        then $\scr{L}_{\leqslant \frak{a}}\subset \scr{L}_{\leqslant \mathfrak{b}}$. 
        \item
        If $\frak{a}\in \scr{I}(V)$ and $n\in \Z$,
        we have 
        $u^n\scr{L}_{\leqslant\frak{a}}=\scr{L}_{\leqslant \frak{a}+2\pi \i n s}$. 
    \end{enumerate}
\end{definition}
Let $\scr{L}_\bullet$
be a wild pre-Stokes filtration on an $\scr{A}_{\per}$-module
$\scr{L}$. 
For each index $\frak{a}\in \scr{I}(V)$,
we set
$(\scr{L}_{<\frak{a}})_{e^{\i\theta}}
=\sum_{\frak{b}<_\theta \frak{a}}
(\scr{L}_{\leqslant \frak{b}})_{e^{\i \theta}}$,
which defines 
a subsheaf $\scr{L}_{<\frak{a}} \subset \scr{L}_{\leqslant \frak{a}}$
and then we set 
$\gr_{\frak{a}}\scr{L}=\scr{L}_{\leqslant{\frak{a}}}/
\scr{L}_{<\frak{a}}$. 
By the third condition
in Definition \ref{preStokes},
for any $U\subsetneq S^1$,
the module $\bigoplus_{\frak{a}\in \scr{I}(U)}\gr_{\frak{a}}\scr{L}$
is naturally equipped with the 
structure of $\C[u,u^{-1}]$-modules. 
Then, by the first condition,
we obtain a sheaf of
$\C[u,u^{-1}]$-modules 
$\gr\scr{L}$ such that we have 
$\gr\scr{L}_{|U}=
\bigoplus_{\frak{a}\in \scr{I}(U)}\gr_{\frak{a}}\scr{L}$
for any $U\subsetneq S^1$. 

The tensor product 
$\scr{A}_\per\otimes_{\C[u,u^{-1}]}\gr\scr{L}$
over $\C[u,u^{-1}]$
is naturally equipped with 
a wild pre-Stokes filtration.  
We will omit the subscript $\C[u,u^{-1}]$ in the following.

\begin{definition}\label{DefSt}
    For a locally free $\scr{A}_{\per}$-module
    $\scr{L}$, a wild pre-Stokes filtration
    $\scr{L}_{\leqslant \bullet}$
    is called a \textit{wild Stokes filtration} if 
    for any $x\in S^1$ 
    there exists a pair $(U,\eta)$ of an open neighborhood 
    $U$ of $x$ and an isomorphism 
    $\eta\colon 
    ( \scr{A}_{\per}\otimes\gr\scr{L})_{|U}
    \longrightarrow \scr{L}_{|U}$
    of $\scr{A}_{\per|U}$-modules
    such that 
    \begin{enumerate}
        \item for any open subset $V\subset U$ and
        $\frak{a}\in\scr{I}(V)$,
        we have 
        $\eta(\gr_{\frak{a}}\scr{L})\subset 
        \scr{L}_{\leqslant \frak{a}}$, and 
        \item the composition $\gr_{\frak{a}}\scr{L}\xrightarrow{\eta}\scr{L}_{\leqslant \frak{a}}\rightarrow
        \gr_{\frak{a}}\scr{L}$ with the quotient map
        is the identity. 
    \end{enumerate}
\end{definition}

A wild Stokes-filtered $\scr{A}_{\per}$-module $\scr{L}$
is called graded if
it is isomorphic 
to 
$\scr{A}_{\per}\otimes \gr\scr{L}$ 
as a filtered $\scr{A}_{\per}$-module.

\subsubsection{Basic notions and operations}
Let $(\scr{L},
\scr{L}_{\leqslant \bullet})$
and $(\scr{L}',
\scr{L}'_{\leqslant \bullet})$
be wild Stokes-filtered 
$\scr{A}_{\per}$-modules. 
Then the tensor 
product
$\scr{L}\otimes \scr{L}'$
and 
the sheaf of internal homomorphisms
$\homscr(\scr{L},\scr{L}')$
are equipped with Stokes filtrations as follows:  
\begin{align*}
    (\scr{L}\otimes \scr{L}')_{\leqslant \frak{a}}
    &\coloneqq\sum_{\frak{b}\in\scr{I}(U)} 
    \scr{L}_{\leqslant\frak{b}}
    \otimes_{\scr{A}_{\per}^{\leqslant 0}}
    \scr{L}'_{\leqslant\frak{a}-\frak{b}},\\
    \homscr(\scr{L},\scr{L}')_{\leqslant\frak{a}}
    &\coloneqq 
    \sum_{\frak{b}\in\scr{I}(U)}
    \homscr_{\scr{A}_{\per}^{\leqslant 0}}(\scr{L}_{\leqslant \frak{b}},
    \scr{L}'_{\leqslant \frak{a}+\frak{b}}).
\end{align*}
We also set 
$\mathscr{E}nd(\scr{L})=\homscr(\scr{L},\scr{L})$. 

\subsubsection{Relation to the mild case}
The Stokes-filtered $\scr{A}_\per$-modules
introduced in \cite{shamoto2022stokes}
can be regarded as a special class of the
wild Stokes-filtered $\scr{A}_{\per}$-modules.
The Stokes-filtered 
$\scr{A}_{\rm per}$-modules 
in \cite{shamoto2022stokes}
are
the wild 
Stokes-filtered $\scr{A}_\per$-modules
such that 
we have 
$\gr_{ \frak{a}}\scr{L}=0$
for any
$\frak{a}=r s\log s+\sum_{j=1}^p c_js^{j/p} $
with $r\neq 0$. 
It might be more appropriate to refer to this class as ``mild" Stokes filtrations, and to call our wild Stokes filtrations simply ``Stokes filtrations" without an adjective. 
\subsubsection{Classification}
Fix a graded wild Stokes-filtered $\scr{A}_{\per}$-module
$\scr{G}$.
We set 
$\mathscr{A}{ut}^{<0}(\scr{G})=\id +\scr{E}nd(\scr{G})_{<0}$
where ``$\id$" denotes the identity endomorphism. 
\begin{lemma}\label{StCl}
    There is a natural one-to-one correspondence 
    between the cohomology set
    $H^1(S^1,\scr{A}{ut}^{<0}(\scr{G}))$
    and the set of isomorphism 
    classes of pairs 
    $(\scr{L},\Xi)$
    of a wild Stokes-filtered 
    $\scr{A}_{\per}$-module
    $\scr{L}$ and an isomorphism
    $\Xi\colon\scr{A}_{\per}\otimes \gr\scr{L}\simeqto\scr{G}$. 
\end{lemma}
\begin{proof}
    The proof is completely parallel to 
    that of \cite{shamoto2022stokes}*{Theorem 3.10}.
\end{proof}

\section{Riemann--Hilbert 
correspondence}\label{Riemann--Hilbert correspondence}
\subsection{Review of 
the theory of difference modules}
In this subsection, 
we shall recall
the classical 
theory of difference modules
mainly following the standard reference
\cite{Galois} to prepare some results 
used in this paper.

\subsubsection{Notation}
A pair $(R,\mathsf{a})$ 
consisting of a commutative ring $R$ and
an automorphism 
$\mathsf{a}\colon R\to R$
is called a difference ring. 
A typical example is a pair 
$(\C[s],\Phi^*)$
of the polynomial ring 
$\C[s]$
and an automorphism $\Phi^*$
defined as $\Phi^*(P(s))=P(s+1)$, $P(s)\in \C[s]$. 
If $R$ is a field, the difference ring is called 
a difference field. 
Let $\C\conv{s^{-1}}$ be 
the field of convergent Laurent series 
in the complex variable $s^{-1}$. 
We have an automorphism
$\phi\colon \C\conv{s^{-1}}\to \C\conv{s^{-1}}$
defined by \[\phi(f)(s^{-1})=f(s^{-1}(1+s^{-1})^{-1}),\]
where we use the expansion
$s^{-1}(1+s^{-1})^{-1}=-\sum_{n>0} (-1)^n s^{-n}$. 
The pair \[(\C\conv{s^{-1}},\phi)\]
is the main example of 
a difference field in this 
paper. 
A morphism between two difference rings
is a morphism of rings that is compatible
with the automorphisms. 

Let $(R,\mathsf{a})$ 
be a difference ring. 
A pair $(M,\mathsf{b})$ consisting of
an $R$-module $M$
and an automorphism
$\mathsf{b}\colon M\to M$
is called a difference module over $(R,\mathsf{a})$ if 
$\mathsf{b}(rm)
=\mathsf{a}(r)\mathsf{b}(m)$ holds
for any $r\in R$ and $m\in M$. 
For two difference modules
$(M_i,\mathsf{b}_i)$ $(i=1,2)$,
a morphism of difference modules 
from $(M_1,\mathsf{b}_1)$ 
to $(M_2,\mathsf{b}_2)$
is a morphism $f\colon M_1\to M_2$ 
of $R$-modules
such that 
$f\circ \mathsf{b}_1
=\mathsf{b}_2\circ f$. 
The set of morphisms is
denoted by $\Hom_{(R,\mathsf{a})}(M_1,M_2)$. 
The subscript $(R,\mathsf{a})$ is often omitted. 
The tensor product 
$M_1\otimes_R M_2$ is naturally equipped
with the structure of 
a difference module by 
$\mathsf{b}_1\otimes \mathsf{b}_2$.
The space of 
$R$-linear maps 
$\Hom_R(M_1,M_2)$
is also regarded as a 
difference module
by $f\mapsto 
\mathsf{b}_2\circ 
f\circ \mathsf{b}_1^{-1}$. 
This difference module is denoted by
$\mathcal{H}om(M_1,M_2)$. 
If $M=M_1=M_2$, this module is 
also denoted by $\mathcal{E}nd(M)$. 
The direct product $M_1\oplus M_2$
is also defined in a natural way. 
\subsubsection{Elementary examples}
\begin{example}\label{reg}
    For a positive integer 
$r$,
we take a constant matrix
$A\in \End(\C^r)$.
We then define a
difference module $\scr{R}_A=
(\C\conv{s^{-1}}^{\oplus r},\psi_A)$
over $(\C\conv{s^{-1}},\phi)$
as follows:
$\psi_A=(1+s^{-1})^A\phi^{\oplus r}$
where we put 
$(1+s^{-1})^A=\exp(A\log (1+s^{-1}))$.
\end{example}

Fix a positive integer $p$. 
Let $s^{1/p}$ denote the 
symbol for the $p$-th root of $s$. 
In other words, 
we consider 
an extension
$\C\conv{s^{-1}}\subset \C\conv{s^{-1/p}}$
by sending $s^{-1}$ to $(s^{-1/p})^p$. 
The field $\C\conv{s^{-1/p}}$
admits the structure of 
a difference field by
an automorphism defined as
$\phi_p(s^{-1/p})=s^{-1/p}(1+s^{-1})^{-1/p}$.
The inclusion
$\C\conv{s^{-1}}\subset \C\conv{s^{-1/p}}$
defined above is a morphism of difference
fields. 
\begin{example}\label{ExEXP}
    For $q\in \Z$ and 
    $c_1,\dots,c_p\in \C$, 
    we set 
    $\frak{a}(s)= 
    qs\log (s^{1/p})+
    \sum_{i=1}^p c_is^{i/p}$. 
    Although this is a multi-valued function,
    the associated function
    $\exp(\phi_p(\frak{a})-\frak{a})$
    is naturally regarded as an element of 
    $\C\conv{s^{-1/p}}$. 
    We define a
    difference module 
    $\scr{E}^{\frak{a}}=
    (\C\conv{s^{-1/p}},\psi_{\frak{a}})$
    by 
    $\psi_{\frak{a}}=
    \exp(\phi_p(\frak{a})-\frak{a})
    \phi_p$. 
    We note that 
    $\frak{a}$ 
    is also regarded as 
    a local section of $\scr{I}$
    defined in \S \ref{Index}
    if we fix a branch of $s^{1/p}$.  
\end{example}

\subsubsection
{Formal decomposition theorem}
Let $\C\pole{s^{-1}}$ denote the field of 
formal Laurent series. 
We set 
$\widehat{\phi}(f)(s^{-1})
=f(s^{-1}(1+s^{-1})^{-1})$ for 
$f\in \C\pole{t}$ similarly as above. 
Then the pair $(\C\pole{s^{-1}},\widehat{\phi})$ is 
a difference field and the natural inclusion
$\C\conv{s^{-1}}\subset \C\pole{s^{-1}}$ is a morphism
of difference fields.
For a positive integer $p$,
we also define the automorphism
$\widehat{\phi}_p$ on $\C\pole{s^{-1/p}}$
similarly as $\phi_p$. 
The following theorem is well-known
(see \cite{Galois} and references therein): 
\begin{theorem}
     For any difference module
     $\widehat{\scr{M}}$ over 
     $(\C\pole{s^{-1}},\phi)$, 
     there exist 
     a positive integer $p$  
     and an isomorphism 
     \begin{align}\label{Xihat}
         \widehat{\Xi}\colon
         \widehat{\scr{M}}
         \otimes\C\pole{s^{-1/p}}
         \simeqto
         \bigoplus_{j=1}^m
         (\scr{E}^{\frak{a}_j}
         \otimes
         \scr{R}_{A_j})
         \otimes
         \C\pole{s^{-1/p}}
     \end{align}
     where
     $\scr{E}^{\frak{a}_j}$ is the module in 
     Example
     $\ref{ExEXP}$
     for $\frak{a}_j
     =q_js\log (s^{1/p})+\sum_{i=1}^p c_{j,i}s^{i/p}$
     and $\mathscr{R}_{A_j}$ 
     is a module in 
     Example $\ref{reg}$ 
     for a matrix
     $A_j\in \End(\C^{r_j})$. 
     Note that we have 
     $\dim_{\C\pole{s^{-1}}}\widehat{\scr{M}}
     =\sum_j r_j $. \qed
\end{theorem}

\begin{remark}
The smallest $p$ with such a 
decomposition is called the 
ramification index of $\widehat{\scr{M}}$.
If $p$ is the ramification index, 
the set of modules $\{(\mathscr{E}^{\frak{a}_j},\mathscr{R}_{A_j})\}_j$
and the isomorphism $\widehat{\Xi}$
in \eqref{Xihat}
are unique up to trivial identifications. 

A difference module $\mathscr{M}$
is called mild 
if its completion $\widehat{\mathscr{M}}$
admits a decomposition as in \eqref{Xihat} 
where $q_j = 0$ for all $j$ in the expressions $\frak{a}_j =q_js\log (s^{1/p})+\sum_{i=1}^p c_{j,i}s^{i/p}$.
Difference modules that are not necessarily mild are called wild.
\end{remark}
\begin{remark}\label{Rmkgraded}
Let $\scr{M}$ be a difference module 
over $\C\conv{s^{-1}}$ and 
set $\widehat{\scr{M}}=\scr{M}\otimes 
\C\pole{s^{-1}}$. We have 
the formal decomposition 
as in \eqref{Xihat}. 
Then there exists 
a difference module 
$\scr{N}$ over $\C\conv{s^{-1}}$
such that 
\begin{align*}
\scr{N}\otimes 
\C\conv{s^{-1/p}}&\simeq 
\bigoplus_j\scr{E}^{\frak{a}_j}\otimes \scr{R}_{A_j} \text{ and} \\
\scr{M}\otimes\C\pole{s^{-1}}&\simeq 
\scr{N}\otimes\C\pole{s^{-1}}.
\end{align*}
The module $\scr{N}$
is called a graded module of $\scr{M}$.
\end{remark}

\subsection{The sheaves 
$\mathscr{P}$ and $\mathscr{Q}$}
Let 
$\widetilde{\phi}\colon 
\widetilde{\scr{O}}\to 
\widetilde{\scr{O}}$
denote the automorphism
defined as $\widetilde{\phi}(f)(s)=f(s+1)$
for a local section $f\in \widetilde{\mathscr{O}}$
(see \cite{shamoto2022stokes}*{\S 2.2.2} for precise definition).
In this subsection, we shall define 
two $\widetilde{\phi}$-invariant subsheaves of rings
$\mathscr{P}\subset \widetilde{\mathscr{O}}$
and $\mathscr{Q}\subset \mathscr{P}$. 
The sheaf $\mathscr{Q}$ and its variants
will be used to define the 
Riemann--Hilbert functor. 
Recall that we put $u=\exp(2\pi \i s)$.
We construct  
$\scr{Q}$
to satisfy two requirements:
\begin{itemize}
    \item It should contain
    solutions for wild
    difference modules, such as $s^s$. 
    \item It should filter out the 
    extraneous periodic functions
    such as $\exp(u+u^{-1})$. 
\end{itemize}
To satisfy the requirements,
we impose growth conditions on the functions when $\mathrm{Re}(s)$
is bounded and $\mathrm{Im}(s)\to\pm\infty$.
Such conditions make sense only when 
the functions are globally defined. 
Hence we first construct the sheaf 
$\scr{P}\subset \widetilde{\scr{O}}$
of globally defined functions in \S \ref{SP}.
We then impose the growth condition 
to define $\scr{Q}\subset \scr{P}$ in \S \ref{SGC}.  
\subsubsection{An invariant subsheaf 
of globally defined functions in 
$\widetilde{\scr{O}}$}\label{SP}
We shall define a $\widetilde{\phi}$-invariant subsheaf 
$\mathscr{P
}$ of 
$\widetilde{\mathscr{O}}$ 
in the following. 
For a compact subset 
${K}\subset \C$, let
${K}_+$, ${K}_-$, and 
${K}_\Z$ be defined as 
\[
{K}_\pm = 
\{z=s\pm n\mid s\in
{K}, n\in \Z_{\geqslant 0}  \}
\]
and ${K}_\Z={K}_+\cup {K}_-$.
Set $U_{{K},\mp}=\C\setminus {K}_\pm$ 
and $U_{{K},\Z}=\C\setminus {K}_\Z$. 
For any pair of open subsets
$V\subset S^1\setminus \{1\}$ and 
$\widetilde{V}\subset \widetilde{\C}$
with $V=S^1\cap \widetilde{V}$, 
we have 
$\widetilde{V}\cap U_{{K}_-}\neq \emptyset$
and the intersection of the closure of 
$\widetilde{V}\cap U_{{K}_-}$ 
in $\widetilde{V}$
and $S^1=\partial \widetilde{\C}$ 
coincides with $V$. 
It follows that 
we have a natural restriction 
map $\mathscr{O}
(U_{{K},-})\to 
\widetilde{\scr{O}}(V)$.
In a similar way, 
if we put 
\begin{align}\label{UKV}
    U_{{K},V}\coloneqq \begin{cases}
    \C\setminus {K}
    &(\text{if }1\in V, -1\in V)\\
    U_{{K},+}
    &(\text{if }1\in V, -1\notin V)\\
    U_{{K},-}
    &(\text{if }1\notin V, -1\in V)\\
    U_{{K},\Z}
    &(\text{if }
    1\notin V, -1\notin V), 
    \end{cases}
\end{align}
for an open subset $V\subset S^1$, 
then we have the restriction map 
$\scr{O}(U_{K,V})
\longrightarrow \widetilde{\scr{O}}(V)$. 
Let $\mathscr{P}_K(V)$
denote the image of this map.

\begin{definition}
We define the subsheaf $\scr{P}\subset \widetilde{\scr{O}}$
by the following condition: 
for any connected open subset $V\subset S^1$, we have 
\begin{align*}
    \mathscr{P}(V)
    = \sum_{{K}\subset \C_s\colon  \text{compact}} 
    \mathscr{P}_{{K}}(V)\subset \widetilde{\scr{O}}(V)
\end{align*}
where  
the sum runs over all compact subsets 
${K}\subset \C_s$.
\end{definition}
In other words, $\scr{P}(V)$ consists of  sections
in $\widetilde{\O}(V)$
which are represented by a holomorphic functions
on $U_{{K},V}$ for some compact subset 
${K}$. 
\subsubsection{Growth conditions}
\label{SGC}
     
    \begin{definition}\label{mod u}
    Let $h$ be 
    a holomorphic function defined on 
    a domain which  contains the closed region
    $\{s\in\C\mid \mathrm{Im}(s)\geq R\}$ 
    for some $R\in \R$. 
    For a compact subset $I\subset \R$
    and $R'>R$,
    we set $\mathbb{V}_I=\{s\mid \mathrm{Re}(s)\in I,
        \mathrm{Im}(s)\geq R\}$,
    which we call a vertical strip on $I$. 
        Then $h$ is said to be \textit{of moderate growth 
        in  $u$} 
        if there exists a positive integer 
        $N$ such that the function
        $|u^N h(s)|$ decays rapidly in $s$ 
        on  
        $\mathbb{V}_I$ 
        for any compact subset $I\subset \R$ and $R'\gg0$. 
    \end{definition}
    We define the notion of 
    `\textit{of moderate growth in $u^{-1}$}'
    in a similar way 
    for holomorphic functions 
    defined on domains
    which contain
    $\{s\in\C\mid \mathrm{Im}(s)\leq R\}$
    for some $R\in\R$. 
    For a connected open subset 
    $V\subset S^1$, 
    a section of $\scr{P}(V)$ is 
    represented by a holomorphic function
    $h$ on $U_{{K},V}$
    for a compact subset 
    ${K}\subset \C_s$.
    
    Then, in each case, the domain $U_{K,V}$ defined in \eqref{UKV} contains
    $\{s\mid \mathrm{Im}(s)\geq R\}$
    and 
    $\{s\mid \mathrm{Im}(s)\leq-R\}$
    for sufficiently large $R>0$. 
    By the  
    identity theorem, 
    if $V\cap (0,\pi)_s\neq\emptyset$
    (resp. $V\cap(-\pi,0)_s\neq \emptyset$)
    and 
    $h$ is of moderate growth in $u$
    (resp. $u^{-1}$),
    then any other representative of
    the section in $\scr{P}(V)$
    is also of moderate growth in $u$
    (resp. $u^{-1}$).
    A section  of 
    $\scr{P}(V)$
    is said to be of moderate growth in $u$
    (resp. $u^{-1}$)
    if a representative 
    of the section is of moderate growth 
    in $u$ (resp. $u^{-1}$). 
        
\begin{definition} 
   For a connected open subset
   $V\subset S^1$,
   we define a subring $\mathscr{Q}(V) \subset \scr{P}(V)$ consisting of the sections that are
   of moderate growth in $u$ if 
   $V\cap (0,\pi)_s\neq \emptyset$
   and 
   of moderate growth in $u^{-1}$
   if $V\cap (-\pi,0)_s\neq \emptyset $.
   The associated sheaf is denoted by
   $\scr{Q}$. 
\end{definition}
\begin{remark}
    If $V$ in the definition above 
    contains $1$ or $-1$, i.e. if $V$ 
    satisfies
    $V\cap(-\pi,0)\neq \emptyset$
    and $V\cap(0,\pi)\neq \emptyset$,
    then the sections in $\scr{Q}(V)$
    are assumed to be of moderate growth 
    both in $u$ and in $u^{-1}$. 
\end{remark}
Clearly, the function
$\exp(u+u^{-1})$
is not contained in $\scr{Q}(U)$
for any open $U\subset S^1$. 
\begin{example}
    [$s^s\in \scr{Q}$]
    Fixing a branch of 
    $\log s$,
    we may see 
    $s^s=\exp (s\log s)$
    as a local section of 
    $\mathscr{Q}$.
    Indeed, we have 
    $|s^s|=\exp(\mathrm{Re}(s)\log|s|-\mathrm{Im}(s)\arg(s))$
    and hence $u^\pm$-moderate growth. 
\end{example}
\subsubsection{Filtration}\label{filtration}
Let $\scr{A}^{<0}$ 
(resp. 
$\scr{A}^{\leqslant 0}$)
be the sheaf of germs of functions of rapid decay (resp. moderate growth) in 
$\widetilde{\scr{O}}$. 
Here, a local section 
    $f\in \widetilde{\mathscr{O}}$ on $U\subset S^1$, represented by a holomorphic function 
    $\widetilde{f}\in \widetilde{\jmath}_*\scr{O}_{\C^*}$ on an open subset $\widetilde{U}\subset\widetilde{\C}_z$ with 
    $U=S^1\cap \widetilde{U}$ is \textit{of moderate growth}
    if for any compact subset $K\subset \widetilde{U}$, there exist constants $C_K>0$
    and $N_K\geq 0$ such that 
    \[ |f(z) |\leq C_K|z|^{-N_K} \text{ for any }
    z\in K\setminus U.
    \]
Similarly, $f$ is \textit{of rapid decay}
if for any compact $K\subset \widetilde{U}$
and any $N\geq 0$, there exists a constant 
$C_{K,N}>0$ such that 
\[|\widetilde{f}(z)|\leq C_{K,N}|z|^{N}
\text{ for any }z\in K\setminus U.\] 

For an open subset $U\subset S^1$ 
and a section
$\mathfrak{a}\in \scr{I}(U)$,
we set $\mathscr{A}^{\leqslant\frak{a}}
=\exp(\frak{a})\mathscr{A}^{\leqslant 0}_{|U}$
and $\mathscr{A}^{<\frak{a}}=
\exp({\frak{a}})\mathscr{A}^{<0}_{|U}$. 
We then set 
$\scr{Q}^{\leqslant \frak{a}}=\mathscr{A}^{\leqslant \frak{a}}\cap \mathscr{Q}_{|U}$
and
$\scr{Q}^{< \frak{a}}=
\mathscr{A}^{< \frak{a}}\cap \mathscr{Q}_{|U}$.

\subsubsection{Asymptotic expansion}
For $p\geq 1$, let 
$\mathscr{A}^{(p)}$
denote the subsheaf of 
$\widetilde{\mathscr{O}}$
whose local sections have 
asymptotic expansion in 
$\C\jump{s^{-1/p}}=\C\jump{z^{1/p}}$. 
Here, a local section 
    $f\in \widetilde{\mathscr{O}}$ on $U\subsetneq  S^1$, represented by a holomorphic function 
    $\widetilde{f}\in \widetilde{\jmath}_*\scr{O}_{\C^*}$ on an open subset $\widetilde{U}\subset\widetilde{\C}_z$ with 
    $U=S^1\cap \widetilde{U}$ 
has $\widehat{f}=\sum_{k\geq 0}a_kz^{k/p}$
as its asymptotic expansion 
for a fixed branch of $z^{1/p}$
on $U$
if for any compact subset $K\subset \widetilde{U}$
and for any $N\geq  0$,
there exists a constant $C_{K,N}
$ such that 
\begin{align*}
    |\widetilde{f}(z)-\widehat{f}^{[N]}(z)|
    \leq C_{K,N}|z|^{N/p}
\end{align*}
where $\widehat{f}^{[N]}(z)=\sum_{k=0}^{N-1}a_kz^{k/p}$. 
We set 
$\mathscr{Q}^{(p)}=
\mathscr{A}^{(p)}\cap \mathscr{Q}$.
For an open subset 
$U\subsetneq S^1$, 
fixing a branch of $z^{1/p}$ on $U$,
we have a morphism 
$\asy_{U}\colon \mathscr{Q}^{(p)}_{|U}\to \C\jump{z^{1/p}}_{U}$
of sheaves of rings
given by the asymptotic expansion,
where $\C\jump{z^{1/p}}_{U}$
denotes the constant sheaf 
with fibers $\C\jump{z^{1/p}}$. 
\subsubsection{Vanishing theorem}
The following theorem is used 
implicitly in the proof 
of fully faithfulness of the Riemann--Hilbert functor
defined below. 
\begin{theorem}\label{Q-vanish}
    We have 
    $H^1(S^1,\scr{Q}^{\leqslant 0})=0$. 
\end{theorem}
\begin{proof}
    Let $\scr{A}^{(p,q)}_{\rm mod}$
    denote the sheaf of smooth $(p,q)$
    forms on $\C_s$ restricted to
    the circle $S^1$ at infinity.
    Let $\mathscr{Q}^{(p,q)}_{\rm mod}$
    denote the sheaf of smooth $(p,q)$-forms 
    on $\mathbb{C}_s$ moderate growth at infinity
    which are globally defined 
    and have growth conditions in $u$ and $u^{-1}$
    at infinity similarly defined as in the definition of 
    $\mathscr{Q}$. 
    We have exact sequences
    \begin{align*}
        &0\longrightarrow \mathscr{Q}^{\leqslant0}\longrightarrow
        \scr{Q}^{(0,0)}_{\rm mod}\xrightarrow{\quad \bar{\partial}\quad }\scr{Q}_{\rm mod}^{(0,1)}= \scr{Q}^{(0,0)}_{\rm mod}d\bar{s}\longrightarrow0,\\
        &0\longrightarrow\mathscr{A}^{\leqslant 0}
        \longrightarrow\scr{A}_{\rm mod}^{(0,0)}\xrightarrow{\quad \bar{\partial}\quad }\scr{A}_{\rm mod}^{(0,1)}=\scr{A}^{(0,0)}_{\rm mod }d\overline{s}\longrightarrow0.   
    \end{align*}
    The first sequence gives a soft resolution of $\scr{Q}^{\leqslant 0}$.
    Then, we have 
    \[H^0(S^1,\mathscr{A}_{\rm mod}^{(0,0)})=
    H^0(S^1,\mathscr{Q}_{\rm mod}^{(0,0)}).\]
    Since the vanishing 
    $H^1(S^1,\mathscr{A}^{\leqslant 0})=0$
    is well known (see, for instance, \cite{Sabbah}*{Proposition 8.7, Remark 8.8},
    or references cited in \cite{Sabbah}*{\S 5.2}),
    we obtain the theorem.  
\end{proof}
\subsection{Some technical results}
\label{technical}
In this subsection, we list
some technical results, which
will be proved in \S \ref{Proof Q}. 
Those are $\mathscr{Q}$ versions of
the well-known results in
the classical theory of
difference modules. 

Let $p$ be a positive integer. 
Let $\mathscr{M}$ be a difference module 
over $(\C\conv{s^{-1/p}},\phi_p)$. 
Let $\varpi\colon \widetilde{\C}\to \mathbb{C}$ be the projection
of the real-oriented blowing up, and $\varpi\colon S^1\to \{0\}$
be its restriction. 
We set $\mathscr{M}_{S^1}=\varpi^{-1}\mathscr{M}$,
which is a constant sheaf on $S^1$. 

If $p=1$, the tensor product 
$\scr{Q}\otimes \scr{M}_{S^1}$
over $\C\conv{s^{-1}}_{S^1}$
is naturally equipped with 
the structure of sheaf of 
difference modules
over $(\scr{Q},\widetilde{\phi})$
with difference operator 
$\widetilde{\psi}=\widetilde{\phi}\otimes 
\psi$.
If $p>1$, on any open subset
$U\subsetneq S^1$,
we fix a branch of 
$s^{1/p}$ on $U$ and consider 
a tensor product 
$\mathscr{Q}_{|U}\otimes \scr{M}_{S^1|U}$ over
$\C\conv{s^{-1/p}}$.
The associated difference operator 
is also denoted by $\widetilde{\psi}$. 

Similar notations will 
be used when we replace
$\scr{Q}$
with its subsheaves on a domain
such as $\scr{Q}^{(p)}$
(tensor product over $\C\{s^{-1/p}\}$),
$\scr{Q}^{\leqslant \frak{a}}$,
and $\scr{Q}^{<\frak{a}}$
for some $\frak{a}\in \scr{I}(U)$. 

\subsubsection{Analytic lift}
We show the following theorem in \S
\ref{Existence of analytic}. 
\begin{theorem}[Analytic lift]\label{AL}
 Let $\scr{M}$
  be a difference module over $\C \conv{s^{-1}}$. 
 Let $\scr{N}$ be a graded module of $\scr{M}$
 as in Remark \ref{Rmkgraded}. 
 Let $p> 0$ be the ramification index. 
    Set $\widehat{\scr{M}}
    =\scr{M}\otimes \C\pole{t^{1/p}}$ and $\widehat{\scr{N}}
    =\scr{N}\otimes \C\pole{t^{1/p}}$. 
    Fix 
    an isomorphism
    $\widehat{\Xi}\colon 
    \widehat{\scr{M}}\simeqto 
    \widehat{\scr{N}}$. 
    Then, for each point $x\in S^1$ 
    and for any choice of the branch $s^{1/p}$ around $x$, 
    there exists an open neighborhood $U$ of $x$ and an isomorphism \[{\Xi_U}\colon \scr{Q}_{|U}^{(p)}\otimes \scr{M}_{S^1|U}    \xrightarrow{\  \sim \ }     \scr{Q}^{(p)}_{|U}    \otimes \scr{N}_{S^1|U}\] such that completion $\id_{\C\jump{s^{-1/p}}_{U}}     \otimes_{\scr{Q}_{|U}^{(p)}}     \Xi_U$ via $\asy_U$ coincides with $\widehat{\Xi}$. 
\end{theorem}
This is essentially proven in the classical theory of
difference modules \cite{Galois}*{Theorem 11.4, Lemma 11.8, Theorem 11.10}. 
The technical difficulty here is to 
ensure that
the morphism $\Xi$
is defined over $\mathscr{Q}^{(p)}$.
We first show that the asymptotic
condition in 
the quadrant theorem 
\cite{Galois}*{Theorem 11.4} 
implies the growth condition 
in $u$ or $u^{-1}$ 
in the definition of $\scr{Q}$. 
We then check that 
the periodic functions appearing in the 
proof of \cite{Galois}*{Lemma 11.8,
Theorem 11.10}
are actually in $\scr{Q}_{\per}$,
reviewing the proofs in our notation. 
See \S \ref{Analytic lift} for details. 

\subsubsection{Analytic classification}
A difference module 
$\scr{N}$ over $\C\conv{s^{-1}}$
is called 
\textit{graded}
if $\scr{N}\otimes \C\conv{s^{-1/p}}$
is 
isomorphic to a direct product of 
$\scr{E}^{\frak{a}_j}\otimes\scr{R}_{A_j}$ for some $\frak{a}_j
     =q_js\log (s^{1/p})+\sum_{i=1}^p c_{j,i}s^{i/p}$
and $A_j\in \End(\C^{r_j})$.
The smallest integer $p$
with such an isomorphism 
is called the ramification index of $\scr{N}$.
For a difference module 
$\scr{N}$ over $\C\conv{s^{-1}}$,
let 
$\scr{E}nd_{\scr{Q}}(\scr{N})_{\leqslant 0}$
denote the sheaf of 
local sections $\mathsf{s}$ of $\scr{Q}^{\leqslant 0}\otimes \mathcal{E}nd(\scr{N})_{S^1}$
such that $\widetilde{\psi}(\mathsf{s})=\mathsf{s}$.
We define 
$\mathscr{E}nd_{\scr{Q}}(\scr{N})_{<0}$ in a 
similar way. Then we set
\[\aut^{<0}_{\mathscr{Q}}(\scr{N})
=\id +\mathscr{E}nd_{\scr{Q}}(\scr{N})_{<0}
\subset \scr{E}nd_{\scr{Q}}(\scr
{N})_{\leqslant 0}. \]
This is a sheaf of groups. 
The following 
is a direct consequence of Theorem \ref{AL}. 
\begin{theorem}[Classification]\label{CL}
    Let $\scr{N}$ be a 
    graded difference module. 
    Let $p$ be the ramification index of $\cal{N}$. 
    Then there is a one-to-one correspondence 
    between
    the set of equivalence classes of
    pairs $(\scr{M},\widehat{\Xi})$ of 
    a difference module 
    $\mathscr{M}$ over $(\C\conv{s^{-1}},\phi)$
    and an isomorphism
    \[\widehat{\Xi}\colon 
    \scr{M}\otimes 
    \C\pole{s^{-1/p}}
    \simeqto
    \scr{N}\otimes 
    \C\pole{s^{-1/p}}\]
    of difference modules over 
    $(\C\pole{s^{-1/p}},\phi)$
    and the set
    $H^1(S^1,\aut^{<0}_{\mathscr{Q}}(\scr{N}))$.
\end{theorem}
\begin{proof}
    The proof of this theorem
    is parallel to the classical proof \cite{Galois} 
    if we assume Theorem \ref{AL}.
    We recall the construction 
    of the cohomology class.
    For a pair
    $(\scr{M},\widehat{\Xi})$,
    we obtain an open cover
    $\cal{U}=\{U_k\}_{k\in \Z/K\Z}$
    of $S^1$ by connected open subsets 
    $U_k$ such that 
    $U_k\cap U_\ell=\emptyset$
    if $\ell\notin \{k,k-1,k+1\}$ 
    and analytic lifts 
    \[\Xi_k\colon \scr{M}\otimes\mathscr{Q}_{|U_k}^{(p)}\to \mathscr{N}\otimes \scr{Q}_{|U_k}^{(p)}\]
    such that $\Xi_k\otimes \id_{\C\pole{s^{-1/p}}}=\widehat{\Xi}$
    by Theorem \ref{AL}.
    Then,
    $(U_k\cap U_{k+1},\Xi_k\circ \Xi_{k+1}^{-1})_{k\in \Z/K\Z}$
    defines a cohomology class in 
    $H^1(S^1,\aut^{<0}_{\mathscr{Q}}(\scr{N}))$. 
\end{proof}

\subsubsection{Inhomogeneous difference equation}
 
    \begin{theorem}        
    \label{Claim}
    Let $\scr{M}$ be a difference module 
    over $\C\conv{s^{-1}}$. 
    Let $U\subsetneq  S^1$ be an 
    open interval. 
    Assume that 
    $(\mathscr{Q}^{\leqslant 0} \otimes \scr{M}_{S^1})_{|U}\simeq 
\mathscr{Q}^{\leqslant 0}_{|U} \otimes (\scr{E}^{\frak{a}}\otimes \scr{R}_{G})_{S^1|U}$
for 
$\frak{a}(s)= 
    qs\log (s^{1/p})+
    \sum_{i=1}^p c_i s^{i/p}$
    with $q\in \Z$,
    $c_1,\dots,c_p\in \C$,
and for $G=\gamma\id_{\C^r}+N$ with 
$\gamma\in\C$ and a nilpotent 
matrix $N\in \End(\C^r)$. 
Then, for any ${f}\in 
(\mathscr{Q}^{\leqslant 0} \otimes \scr{M}_{S^1})_{e^{\i\theta}}$ with 
$e^{\i\theta}\in U$,
there exists
$\Lambda_\theta({f}) \in (\mathscr{Q}^{\leqslant 0} \otimes \scr{M}_{S^1})_{e^{\i\theta}}$
such that  
${\widetilde{\psi}}
\Lambda_\theta({f})-\Lambda_\theta({f})={f}$.
Moreover,
if we have ${f}\in (\scr{Q}^{< 0}\otimes \scr{M}_{S^1})
_{e^{\i\theta}}$, 
then we have $\Lambda_\theta({f})\in
(\scr{Q}^{< 0}\otimes \scr{M}_{S^1})_{e^{\i\theta}}$. 
\end{theorem}
The proof of this theorem will be given in \S \ref{Sol to inhom eq}.

\subsection{The Riemann--Hilbert functor}
In this section,
we prove the main theorem 
of this paper assuming 
the technical results 
stated in \S \ref{technical}. 
\subsubsection{The de Rham complexes}
\label{deRham}
For any difference module 
$\scr{M}=(\scr{M},\psi)$,
we shall define de Rham complexes 
as follows:
\begin{align*}
    \DR(\scr{M})&\coloneqq [\scr{M}\xrightarrow{\psi-\id}\scr{M}],\\
    \widetilde{\DR}(\scr{M})
    &\coloneqq 
    [{\scr{Q}}\otimes\scr{M}_{S^1}
    \xrightarrow{{\widetilde{\psi}}-\id}
    {\scr{Q}}\otimes \scr{M}_{S^1}],\\
    \DR_{\leqslant 0}(\scr{M})
    &\coloneqq
    [{\scr{Q}}^{\leqslant 0}\otimes \scr{M}_{S^1}
    \xrightarrow{{\widetilde{\psi}-\id}}
    {\scr{Q}}^{\leqslant 0}\otimes \scr{M}_{S^1}],
    \text{ and }\\
    \DR_{< 0}(\scr{M})
    &\coloneqq
    [{\scr{Q}}^{< 0}\otimes \scr{M}_{S^1}
    \xrightarrow{{\widetilde{\psi}}-\id}
    {\scr{Q}}^{< 0}\otimes \scr{M}_{S^1}].
\end{align*}
Here, of the complexes
are concentrated in degrees
$0$ and $1$. 
 
\begin{theorem}\label{vanish}
We have
$\scr{H}^1\DR_{<0}(\scr{M})=
\scr{H}^1\DR_{\leqslant 0}(\scr{M})=0$. 
\end{theorem}
\begin{proof}
    For each $e^{\i\theta}\in S^1$
    we shall show that 
    $\scr{H}^1\DR_{\leqslant 0}(\scr{M})_{e^{\i\theta}}=0$. 
    Take a sufficiently small open neighborhood 
    $U\subsetneq  S^1$ of $e^{\i\theta}$. 
    By Theorem \ref{AL},
    we may assume that 
     $(\mathscr{Q}^{\leqslant 0} \otimes \scr{M}_{S^1})_{|U}\simeq 
\mathscr{Q}^{\leqslant 0}_{|U} \otimes (\scr{E}^{\frak{a}}\otimes \scr{R}_{G})_{S^1|U}$
for $\frak{a}\in \scr{I}(U)$
and $G=\gamma\id_{\C^r}+N$ with 
$\gamma\in\C$ and a nilpotent 
matrix $N\in \End(\C^r)$.
Then by Theorem \ref{Claim}, 
for any 
$f\in (\scr{Q}^{\leqslant 0}\otimes \mathscr{M}_{S^1})_{e^{\i\theta }}$,
we have 
$(\widetilde{\psi}-\id)(\Lambda_\theta(f))=f$,
which implies the claim. 
The vanishing 
$\scr{H}^1\DR_{<0}(\scr{M})=0$
can also be proved similarly. 
\end{proof}

\begin{corollary}\label{CorDR}
    Let $\mathscr{M}$ be a 
    difference module 
    over $(\C\conv{s^{-1}},\phi)$. 
    For an open subset $U\subsetneq S^1$ and a local section
    $\frak{a}\in\scr{I}(U)$,    
    we set 
    \begin{align*}
        \DR_{\leqslant \frak{a}}(\scr{M})&\coloneqq
        [{\scr{Q}}^{\leqslant \frak{a}}\otimes \scr{M}_{S^1|U}
    \xrightarrow{{\widetilde{\psi}-\id}}
    {\scr{Q}}^{\leqslant \frak{a}}\otimes \scr{M}_{S^1|U}],
    \text{ and }\\
    \DR_{< \frak{a}}(\scr{M})
    &\coloneqq
    [{\scr{Q}}^{<\frak{a}}
    \otimes \scr{M}_{S^1|U}
    \xrightarrow{{\widetilde{\psi}}-\id}
    {\scr{Q}}^{< \frak{a}}\otimes \scr{M}_{S^1|U}],
    \end{align*}
    where the complexes are
    concentrated in degrees $0$
    and $1$. 
    These complexes have non-zero cohomology only in degree zero. 
    Moreover,  
    the inclusions
    $\DR_{<\frak{a}}(\scr{M})\to 
    \DR_{\leqslant\frak{a}}(\scr{M})
    \to \widetilde{\DR}(\scr{M})_{|U}$
    induce injections 
    \[\scr{H}^0\DR_{<\frak{a}}(\scr{M})\to 
    \scr{H}^0\DR_{\leqslant\frak{a}}(\scr{M})
    \to 
    \scr{H}^0\widetilde{\DR}(\scr{M})_{|U}. \]
    The quotient 
    \[\gr_\frak{a}\scr{H}^0
    \DR(\scr{M})\coloneqq
    \scr{H}^0\DR_{\leqslant\frak{a}}(\scr{M})/\scr{H}^0\DR_{<\frak{a}}(\scr{M})\]
    is a local system of
    finite-dimensional $\C$-vector spaces 
    on $U$.\qed
\end{corollary}
\begin{remark}\label{rmk1}
Let $U\subsetneq S^1$ be an
open subset and $\frak{a}\in \scr
I(U)$. 
For a sheaf of difference 
modules
$\scr{N}$
over $(\scr{Q}_{|U},\widetilde{\phi})$,
we define the complexes
$\widetilde{\DR}(\scr{N})$, $\DR_{\leqslant \frak{a}}(\scr{N})$,
and 
$\DR_{<\frak{a}}(\scr{N})$,
in a similar way. 
These complexes are natural in the 
following sense:
if there exists a morphism 
$\xi\colon \mathscr{N}\to\scr{N}'$
of $(\scr{Q}_{|U},\widetilde{\phi})$-modules,
then we have the morphism of complexes 
$\widetilde{\DR}(\scr{N})\to \widetilde{\DR}(\scr{N}')$
and this correspondence is compatible with the composition. 
Similar statements hold for $\DR_{\leqslant \frak{a}}$ and $\DR_{<\frak{a}}$. 
\end{remark}

A key fact is the following:
\begin{lemma}
    $\scr{H}^0\widetilde{\DR}(\C\conv{s^{-1}},{\phi})
    =\scr{A}_\per$. 
\end{lemma}
\begin{proof}
    As discussed in \cite{shamoto2022stokes}*{Lemma 2.2}, local solutions to $\widetilde{\phi}(f)=f$ are periodic functions, which correspond to holomorphic functions in the variable $u=\exp(2\pi \i s)$. Depending on the location of the stalk, the kernel is either
    $\scr{O}(\C^*_u)$,
    $(j_*\scr{O}_{\C^*_u})_0$,
    or $(j_*\scr{O}_{\C_u^*})_{\infty}$,
    where $j\colon \C^*_u\to \P^1$
    denotes the open embedding. 
    Among these, the sections that are globally defined (i.e., sections of $\scr{P}$) and appropriate moderate growth condition in $u$ or $u^{-1}$ (i.e., sections of $\scr{Q}$) precisely constitute the sections of $\scr{A}_\per$.
    This proves the lemma. 
\end{proof}

\begin{corollary}\label{elementary}Let
$\scr{E}^\frak{a}$ and $\scr{R}_A$ 
and $r$
be as in Examples $\ref{ExEXP}$
and $\ref{reg}$. 
For an open subset $U\subsetneq S^1$,
fix a branch of 
$\log s$ on $U$ and the corresponding element 
$\frak{a}\in \scr{I}(U)$. 
Then, we have
$\scr{H}^0\widetilde{\DR}(\scr{Q}_{|U}\otimes (\scr{E}^\frak{a}\otimes \scr{R}_A)_{S^1|U}) 
=e^{-\frak{a}}s^{-A}\scr{A}_{\per|U}^{\oplus r}$. 
\end{corollary}
\begin{proof}
    Note that $\widetilde{\DR}(\scr{Q}_{|U}\otimes (\scr{E}^\frak{a}\otimes \scr{R}_A)_{S^1|U})$
    is defined in Remark \ref{rmk1}.
    Then the claim 
    follows from the local isomorphism
    $\scr{Q}_{|U}^{\oplus r}
    \simeqto 
    \scr{Q}_{|U}\otimes (\scr{E}^\frak{a}\otimes \scr{R}_A)_{S^1|U}$
    of sheaves of 
    difference modules over 
    $\scr{Q}_{|U}$ defined by 
    $\exp(-\frak{a})s^{-A}$
    and naturality mentioned in Remark \ref{rmk1}. 
\end{proof}
\subsubsection{Riemann--Hilbert functor}
For a difference module 
$\scr{M}$ over 
$(\C\conv{s^{-1}},\phi)$,
set
\begin{align*}
    \scr{H}^0\DR_{\leqslant \bullet}
    (\scr{M})
    =\left\{
    \scr{H}^0\DR_{\leqslant \frak{a}}
    (\scr{M})\middle|
    U\subset S^1\text{: 
    open subset, }
    \frak{a}\in \scr{I}(U)
    \right\}.
\end{align*}

\begin{proposition}
    For any difference module
    $\scr{M}$,
    the pair 
    \[
    \mathrm{RH}(\scr{M})\coloneqq 
    (\scr{H}^0\widetilde{\mathrm{DR}}
    (\mathscr{M}),
    \scr{H}^0
    {\mathrm{DR}}_
    {\leqslant \bullet}
    (\mathscr{M}))\]
    is a wild Stokes-filtered 
    $\scr{A}_{\rm per}$-module.
\end{proposition}
\begin{proof}
  The pair $\RH(\scr{M})$
  is a pre-Stokes filtration by the construction.
  To see that 
  it is a Stokes filtration, 
  we first note that 
  for graded 
  difference module 
  $\scr{N}$,
  $\RH(\scr{N})$
  is a graded wild Stokes structure
  by Corollary \ref{elementary}. 

  Then, for a general difference module
  $\scr{M}$, take its graded module
  $\scr{N}$. Assume that 
  the ramification index of 
  $\scr{M}$ is $p$. 
  By Theorem \ref{AL}
  and Remark \ref{rmk1},
  we have a local isomorphism of 
  $(\scr{Q}^{(p)},\widetilde{\phi})$-modules,
  which induces a
  local isomorphism between
  $\RH(\scr{M})$ and 
  $\RH(\scr{N})$ as filtered sheaves,
  which implies the claim. 
\end{proof}
The naturality of 
the de Rham complexes implies that 
each morphism
$\xi\colon \scr{M}\to \scr{M}'$
of difference modules over 
$(\C\conv{s^{-1}},\phi)$
induces a morphism
$\RH(\xi)\colon \RH(\scr{M})
\to \RH(\scr{N})$
of wild Stokes-filtered 
$\scr{A}_\per$-modules. 
Hence we have the following. 
\begin{corollary}
    We have a functor $\mathrm{RH}$
    from the category $\mathsf{Diffc}$
    of difference modules
    to the category $\mathsf{St}^{\rm wild}(\mathscr{A}_\per)$
    of wild Stokes-filtered 
    $\mathscr{A}_\per$-modules.\qed
\end{corollary}
The main result of 
this paper is the following:
\begin{theorem}\label{main theorem}
    The functor 
    \[\RH\colon \mathsf{Diffc}
    \longrightarrow
    \mathsf{St}^{\rm wild}(\mathscr{A}_\per),
    \quad \mathscr{M}\mapsto 
    \RH(\scr{M})
    \]
    is an equivalence of categories. 
\end{theorem}
\begin{proof}
    We show that the functor 
    $\RH$ is fully faithful and 
    essentially surjective. 
    Since  
    the proof is parallel 
    to the case of mild modules \cite{shamoto2022stokes}, 
    we shall sketch
    the proof. 

    Take $\scr{M},\scr{M}'\in \Diffc$
    and put $\scr{L}=\RH(\scr{M})$
    and $\scr{L}'=\RH(\scr{M}')$. 
    To show that $\RH$ is 
    fully faithful,
    we first show that 
    the morphism 
    \begin{align}\label{LFF}
        \scr{H}^0\DR_{\leqslant 0} (\mathcal{H}om(\scr{M},\scr{M}'))\to 
        \scr{H}\!om (\scr{L},\scr{L}')_{\leqslant 0}
    \end{align}
    is an isomorphism. 
    By Theorem \ref{AL},
    we may assume that 
    $\scr{M}$ and $\scr{M}'$
    are graded. Then the claim can be
    checked by the explicit description 
    of the morphism (see \cite{shamoto2022stokes}*{\S 4.5.2}). 
    Taking the push forward to 
    a point for \eqref{LFF},
    by Theorem \ref{vanish}, 
    Theorem \ref{Q-vanish}
    and the projection formula, 
    we have 
\begin{align*}H^0(S^1,\scr{H}^0\DR_{\leqslant 0} (\mathcal{H}om(\scr{M},\scr{M}')))
    &\simeq
    \mathbb{H}^0(S^1,\DR_{\leqslant 0} (\mathcal{H}om(\scr{M},\scr{M}')))\\
    &\simeq H^0(\DR(\mathcal{H}om(\scr{M},\scr{M}')))\\
    &\simeq \Hom (\scr{M},\scr{M}'). 
    \end{align*}
    We also have
    \[H^0(S^1,\scr{H}\!om (\scr{L},\scr{L}')_{\leqslant 0})\simeq 
\Hom(\scr{L},\scr{L}').\] 
    Hence 
    we obtain 
    the isomorphism 
    $\Hom (\scr{M},\scr{M}')\to 
    \Hom(\scr{L},\scr{L}')$,
    which implies that
    $\RH$ is fully faithful. 
    
    It remains to show that 
    $\RH$ is essentially surjective. 
    If $\scr{G}\in \mathsf{St}^{\rm wild}(\scr{A}_\per)$
    is graded,
    we can directly find 
    a graded difference module 
    $\scr{N}$ such that 
    $\RH(\scr{N})\simeq \scr{G}$. 
    To show the general case,
    for $\scr{L}\in \mathsf{St}^{\rm wild}(\scr{A}_\per)$,
    take its associated graded module
    $\scr{G}=
    \scr{A}_\per\otimes\gr(\scr{L})$
    and $\scr{N}\in\Diffc$ such that 
    $\RH(\scr{N})\simeq \scr{G}$.
    Then by Lemma \ref{StCl},
    $\scr{L}$ induces a 
    cohomology class 
    $[\scr{L}]\in H^1(S^1,\aut^{<0}(\scr{G}))$. 
    Since $\RH$ is fully faithful,
    we have an isomorphism 
    $\aut^{<0}_\scr{Q}(\scr{N})
    \simeq \aut^{<0}(\scr{G})$
    and hence 
    $H^1(S^1,\aut^{<0}_\scr{Q}(\scr{N}))
    \simeq 
    H^1(S^1,\aut^{<0}(\scr{G}))$.
    The class $[\scr{L}]$ corresponds
    to a class $[(\scr{M},\widehat{\Xi})]\in 
    H^1(S^1,\aut^{<0}_\scr{Q}(\scr{N}))$. 
    By construction,
    we have 
    $\RH(\scr{M})\simeq \scr{L}$,
    which implies that 
    $\RH$ is essentially surjective. 
\end{proof}

\subsubsection{An example}
The following is the simplest non-trivial 
example of the theory.
\begin{example}[Gamma function] 
Let $(\scr{E}_\Gamma,\psi_\Gamma)$
be a rank-one difference module 
defined as follows:
$\mathscr{E}_\Gamma=\C\conv{s^{-1}}$
and $\psi_{\Gamma}=s^{-1}\phi$. 
Let $\Gamma(s)$ denote the Gamma function,
which is regarded as an element of 
${\scr{Q}}(S^1\setminus \{-1\})$. 
If we set $u=\exp(2\pi \i s)$, 
$(1-u)\Gamma(s)$ can be regarded as a 
global section of ${\scr{Q}}(S^1)$. 
Then, the $\scr{A}_{\per}$-module 
$\scr{L}_\Gamma\coloneqq \scr{H}^0\widetilde{\DR}(\scr{E}_\Gamma)$
is described as follows: For a
connected open subset 
$U\subsetneq S^1$, 
\begin{align*}
    \scr{L}_{\Gamma|U}=
    \begin{cases}
        \mathscr{A}_{\per |U}\Gamma(s)&
        (U\cap \{e^{\pi\i}\}=\emptyset),\\
        \mathscr{A}_{\per |U}(1-u)\Gamma(s)
        &(U\cap \{e^0\}=\emptyset).
    \end{cases}
\end{align*}
The associated graded module
is given by $\gr\scr{L}_\Gamma\simeq \C[u,u^{-1}]e^{-s}s^{s-1/2}$.
This isomorphism follows directly from 
Stirling's formula: for any $\varepsilon >0$,
 \[\Gamma(s)\sim \sqrt{2\pi}e^{-s}s^{s-1/2}\left(1+\frac{1}{12}s^{-1}+\cdots\right)\quad(-\pi+\varepsilon<\arg s<\pi-\varepsilon)\]
as $s\to \infty$, together with the 
reflection formula: 
$(1-u)\Gamma(s)=-2\pi\i e^{\pi\i s}\Gamma(1-s)^{-1}$.
\end{example}

 \section{Proof of the technical results}
 \label{Proof Q}
In this section, 
we prove the technical theorems
stated in \S \ref{technical}.

\subsection{Existence of analytic lift}\label{Existence of analytic}
In this subsection, 
we prove
Theorem \ref{AL}.
\subsubsection{
}
Let $\P^1$ denote the complex projective line with the
homogeneous coordinate $[z_0:z_1]$. 
Let $\Phi\colon \mathbb{P}^1\to \mathbb{P}^1$
be an automorphism defined by 
$\Phi([z_0:z_1])=[z_0+z_1:z_1]$. 
The infinity $\infty=[1:0]$ is the fixed point 
of $\Phi$. 
The restriction of $\Phi$ to the complex plane 
$\C_s=\{[s:1]\mid s\in\C\}$
is also denoted by $\Phi$. 

Let $(\scr{M},\psi)$ be a 
difference module over 
$(\C\conv{s^{-1}},\phi)$. 
If we take a sufficiently large compact subset 
${K}\subset \C_s$, 
we have a $\mathscr{O}_{\P^1 \setminus {K}}
(*\infty)$-module
$\mathcal{M}$
and an isomorphism 
\begin{align*}
    \Psi_{}\colon \mathcal{M}_{|U}\to 
    (\Phi^*\mathcal{M})_{|U}
\end{align*}
on $U\coloneqq\P^1 \setminus 
({K}\cup \Phi^{-1}({K}))$
such that there is an isomorphism 
$\alpha\colon
\mathcal{M}_\infty\simeqto \mathscr{M}$
under which the germ $\Psi_\infty$ of $\Psi$
at $\infty$ coincides with $\psi$,
i.e. 
$\alpha\circ \Psi_\infty =\psi\circ \alpha$. 

The triple $(\mathcal{M},\Psi,\alpha)$
is concretely constructed as follows: 
Take a 
basis 
$e_1^{\infty},\dots, e_r^\infty$ of $\scr{M}$
over $\C\conv{s^{-1}}$. 
We have 
$\psi(e_j^\infty)=\sum_i a_{ij}(s)e_i^\infty$
for some $a_{ij}^\infty\in \C\conv{s^{-1}}$.
Then,
for sufficiently large compact subset $K$,
we have functions
\[a_{ij}\in \Gamma(\P^1\setminus K,\scr{O}_{\P^1\setminus K}(*\infty))\] each of whose 
Laurent expansions 
at infinity is $a_{ij}^\infty$. 
Set $\mathcal{M}=\bigoplus_{i=1}^r\mathscr{O}_{\P^1\setminus K}(*\infty)e_i,$
where $(e_i)_{i=1}^r$
is a global frame.
We then define 
\[\Psi\left(\sum_jf_j(s)e_j\right)=\sum_{i,j}f_j(s+1)a_{ij}(s)e_i,\]
and $\alpha(e_i)=e^\infty_i$.
For later use, 
we let $A=(a_{ij})_{ij}$
denote the matrix-valued holomorphic 
function, which is meromorphic at infinity. 
We call the triple
$(\mathcal{M},\Psi,\alpha)$
a \textit{representative of} $(\scr{M},\psi)$.  

Let $\scr{N}=(\scr{N},\psi^{\gr})$ be
a graded module of $\scr{M}$.
Take a representative 
$(\mathcal{N},\Psi^\gr,\beta)$
of $(\scr{N},\psi^\gr)$.
More concretely,
we take 
a basis $(\epsilon_j^\infty)_{j=1}^r$
    of $\scr{N}$ over $\C\conv{s^{-1}}$. 
Let 
$B^\infty=(b_{ij}^\infty)$ be the corresponding
representation matrix, i.e.,
$\psi^\gr(\epsilon^\infty_j)=
\sum_ib_{ij}^\infty\epsilon_i^\infty$.
Taking a sufficiently large compact set 
$K$, we have representatives 
\[b_{ij}\in \Gamma(\P^1\setminus K,\mathscr{O}_{\P^1\setminus K}(*\infty))\]
of $b_{ij}^\infty$. 
We set $B=(b_{ij})_{i,j}$,
$\mathcal{N}=\bigoplus_{j}\mathscr{O}_{\P^1\setminus K}(*\infty)\epsilon_j$,
$\Psi^\gr=B(s)\Phi^*$, 
and $\beta(\epsilon_j)=\epsilon_j^\infty$. 
   
      We have a formal isomorphism 
        $\widehat{\Xi}(e_j^\infty)=\sum_i\widehat{F}_{ij}(s)\epsilon_i^\infty$
        with $\widehat{F}_{ij}\in \C\pole{s^{-1/p}}$
        by the assumption.
    The matrix $\widehat{F}(s)=(\widehat{F}_{ij}(s))$
    satisfies the difference equation
    \begin{align}\label{formal sol}
        B(s)\widehat{F}(s+1)=\widehat{F}(s)A(s+1). 
    \end{align}

\subsubsection{}
We recall the quadrant theorem.
A quadrant in $\C_s$ is a subset of the form
\begin{align*}
    Q(v,k)=\{v+re^{\i\vartheta}\mid r\geq 0, 
    k\pi/2<\vartheta<(k+1)\pi/2\},
\end{align*}
where $v\in\C_s$ and $k\in\Z$. 
Set $U(v,k)=Q(v,k)\setminus K\subset \C_s$
for a sufficiently large 
compact subset $K$.
For a formal power 
series
$f(s)=\sum_{n>M}a_ns^{-n/p}\in \C\pole{s^{-1/p}}$,
and an integer $N$,
we set 
$f^{[N]}(s)=\sum_{M<n\leq N}a_ns^{-n/p}$. 
\begin{theorem}\label{Analytic lift}
Fix a branch of $\log s$ on a quadrant 
$Q(v,k)$ with $v\in \C_s$, $k\in \Z$. 
Set $s^{1/p}=\exp(p^{-1}\log s)$. 
    There exists 
    an isomorphism
    \begin{align*}
        \Xi_{(v,k)}\colon \mathcal{M}_{|U(v,k)}\to \mathcal{N}_{|U(v,k)}
    \end{align*}
    of $\scr{O}_{U(v,k)}$-modules 
    with the following properties:
    \begin{itemize}
        \item $\Xi_{(v,k)}$
        is compatible with $\Psi$ and $\Psi^\gr$
        in the sense that
        \begin{align}\label{XiPsi}
            \Xi_{(v,k)} \circ \Psi=\Psi^\gr\circ \Xi_{(v,k)}.
        \end{align}
        \item 
        Let  
        $(e_j)$ and $(\epsilon_i)$ 
        be global frames of $\mathcal{M}$
        and $\mathcal{N}$, respectively. 
        Define the entries 
        $F_{ij}{(v,k;s)}$ by
        \begin{align}\label{Xi_(v,k)}
            \Xi_{(v,k)}(e_{j|U(v,k)})=\sum_iF_{ij}{(v,k;s)}\epsilon_{i|U(v,k)}.
        \end{align}
        Then, for a sufficiently large compact subset  
        $K$ and for each $N$,  
        there exists a positive constant 
        $C>0$ such that we have
        \begin{align}\label{FFhat}
            |F_{ij}(v,k;s)-\widehat{F}^{[N]}_{ij}(s)|<C|s|^{-N/p}
        \end{align}
        for any $i,j$, and 
        $s\in U(v,k)=Q(v,k)\setminus K$. 
    \end{itemize}
    
\end{theorem}
\begin{proof}
    This is an adaption of 
    the classical quadrant theorem 
    \cite{Immink2}*{Theorem 3.2},
    \cite{Galois}*{Theorem 11.3, Theorem 11.4} to our setting. 
    We briefly explain this correspondence. 
    The matrix-valued 
    formal series $\widehat{F}(s)$
    satisfies the difference equation 
    \eqref{formal sol}.
    By \cite{Galois}*{Theorem 11.4},
    there exists
    an analytic solution 
    $F(v,k;s)=(F_{ij}(v,k;s))$
    to the corresponding analytic difference equation (obtained by 
    replacing 
    $\widehat{F}$ by $F$ in \eqref{formal sol} 
    which we denote by \eqref{formal sol}')
    such that 
    the asymptotic expansion of $F$
    is exactly $\widehat{F}$. 
    Then, we set $\Xi_{(v,k)}$
    by \eqref{Xi_(v,k)}. 
    The equation \eqref{formal sol}'
    implies \eqref{XiPsi}.
    The fact that the asymptotic 
    expansion of $F(v,k;s)$ is
    $\widehat{F}$ implies the 
    second condition.
\end{proof}

\begin{corollary}\label{quadrant}
    Theorem $\ref{AL}$ holds for $x\in S^1\setminus \{\i^k\mid k\in \Z\}$. 
\end{corollary}
\begin{proof}
Take 
$k\in \Z$
such that 
the closure of $U(v,k)$
in $\widetilde{\P}^1$
contains an open neighborhood $U$ of $x$. 
We shall show that the morphism
$\Xi_{(v,k)}$ induces the desired morphism $\Xi_U$ in 
Theorem \ref{AL}. 
For simplicity,
we assume that $k=0$
(the other cases can be 
treated in a similar way).

The matrix-valued function $F_v{(s)}=F(v,0;s)$ satisfies the equality 
    \begin{align}\label{'}
        F_v(s+1)=B(s)^{-1}F_v(s)A(s+1) 
    \end{align}
for $s,s+1\in U(v,0)$, 
which was called \eqref{formal sol}'
in the proof of Theorem \ref{Analytic lift}.
Using this equality \eqref{'}, 
the entries of $F_v(s)$
can be 
analytically continued to 
holomorphic functions on 
the Minkowski sum
\[U(v,k)+\Z=\{s\in \C\mid\exists n\in\Z\text{ s.t. } s-n\in U(v,k)\}.\] 
Hence the entries of 
$F_v(s)$
define sections in $\mathscr{P}(U)$. 
These sections
are in fact in $\mathscr{Q}^{(p)}(U)$
by \eqref{FFhat} and \eqref{'}. 
More precisely, 
if $\mathbb{V}_I\subset U(v,0)$,
then \eqref{FFhat}
implies that the entries are of 
moderate growth in $u$ on $\mathbb{V}_I$ (see Definition \ref{mod u} for notations). 
The other cases are reduced to 
this case by the equality \eqref{'}. 
The equality \eqref{FFhat}, or the second condition in Theorem \ref{Analytic lift}, also implies that 
the entries are in $\mathscr{A}^{(p)}(U)$. 

Let $F_{U}=(F_{U,{ij}})$ be the matrix-valued functions whose
entries are those sections in $\mathscr{Q}^{(p)}(U)$.
Let $(e_{U,i})_i$ and 
$(\epsilon_{U,j})$
denote the frames of 
$\mathscr{Q}^{(p)}_{|U}\otimes \mathscr{M}$
and $\mathscr{Q}^{(p)}_{|U}\otimes \mathscr{N}$, associated with the global frames $(e_i)_i$ and $(\epsilon_j)_j$,
respectively. 
Then the morphism 
\[\Xi_U\colon\mathscr{Q}^{(p)}_{|U}\otimes \mathscr{M}\to \mathscr{Q}^{(p)}_{|U}\otimes \mathscr{N};
e_{U,j}\mapsto \Xi_U(e_{U,j})=\sum_iF_{U,ij}\epsilon_i\]
satisfies the conditions in Theorem \ref{AL}. 
\end{proof}
\begin{remark}
    A part of the discussion in the proof 
    of Corollary \ref{quadrant}
    is taken from 
    \cite{MR1091837}*{Corollary of Theorem 1.14}.
    We also note that 
    the open subset $U\subset S^1$
    can be taken as 
    the connected component of 
    $S^1\setminus \{\i^k\mid k\in \Z\}$
    which contains $x$. 
\end{remark}

\subsubsection{}
We recall 
classical theorems 
in \cite{Galois}.
Since the formulation chosen there is slightly—though not essentially—different from ours, we provide the corresponding statements and their proofs translated into our setting.

For two quadrants, 
$U,U'$,
we say that $U$ and $U$
have a sufficiently large 
intersection if the intersection of 
$U\cap U'$ and its shift 
$\{s+1\mid s\in U\cap U'\}$
is non-empty. 
The intersection $U\cap U'$
is called a sufficiently large intersection
of two quadrants. 

Let $\mathcal{N}=(\mathcal{N},\Psi^\gr,\beta)$ be as above. 
Let $(\epsilon_{j})$ be a global frame of $\mathcal{N}$. 
Let $Y$ be a quadrant or
sufficiently large intersection of
two quadrants. 
A morphism 
$\tau\colon \mathcal{N}_{|Y}\to \mathcal{N}_{|Y}$
is called asymptotic to the identity
on an open subset $U\subset Y$
if the diagonal entries
(resp. off diagonal entries) of the presentation matrix 
of $\tau$ with respect to $(\epsilon_{j})$ are asymptotic to $1$
(resp. $0$) on the region. 
The morphism $\tau$ is said to 
satisfy $u^{\pm 1}$-moderate growth 
condition if its entries are of 
moderate growth in $u$ or $u^{-1}$
(in the sense of Definition \ref{mod u})
on every vertical strip contained in $Y$. 
These conditions are independent of
the choice of the global frame $(\epsilon_j)$. 

\begin{theorem}[{\cite{Galois}}]
\label{vert}
Let $k$ be an integer.
For $\ell \in \{k-1,k\}$, 
take quadrants $U_\ell=U(v_\ell,\ell)$
such that the intersection 
$V_k=\bigcap_{\ell\in \{k,k-1\}}U_\ell$
is sufficiently large. 
    Let $\tau\colon \mathcal{N}_{|V_k}
    \simeqto \mathcal{N}_{|V_k}$
    be an automorphism 
    on 
    $V_k$ which 
    is asymptotic to the identity
    on $V_k$. 
    Then there exists a pair of automorphisms 
    $\tau_\ell\colon \mathcal{N}_{|U_\ell}\to \mathcal{N}_{|U_\ell}$
    $(\ell \in \{k-1,k\})$ 
    satisfying the following conditions$:$
    \begin{itemize} 
        \item We have the equality $\tau=\tau_{k|V_k}^{-1}\circ \tau_{k-1|V_k}$. 
        \item 
        $\tau_k$ is asymptotic 
        to the identity on $U_k$. 
        \item 
        $\tau_{k-1}$
        is asymptotic to the identity
        on $U_{k-1}(\varepsilon)$
        for every $\varepsilon>0$
        and satisfies the 
        $u^{\pm 1}$-moderate growth condition
        where we set$:$
        \begin{align*}
            U_{k-1}(\varepsilon)=\{s\in U_{k-1}\mid 
            \pi (k-1)/2+\varepsilon<\arg(s)< k\pi/2\}.
        \end{align*}
    \end{itemize}
\end{theorem}
\begin{remark}
   The proof 
   of this theorem is 
   essentially provided 
   in 
   \cite{Galois}*{Lemma 11.8} when 
   $k\equiv 0 \mod 2$
   and in 
   \cite{Galois}*{Theorem 11.10} when
   $k\equiv 1\mod 2$.
\end{remark}

\begin{corollary}\label{i}
    Theorem $\ref{AL}$ holds for $x\in  \{\i^k\mid k\in\mathbb{Z}\}$. 
\end{corollary}
\begin{proof}
    Take $k\in \Z$ with  
    $x=\i^k$. 
    Take quadrants 
    $U_\ell=U(v_\ell,\ell)$ 
    with $v_\ell\in\C$ for $\ell\in \{k,k-1\}$
    such that the intersection
    $V_k=U_{k}\cap U_{k-1}$
    is sufficiently large. 
    Let $\Xi_{\ell}=\Xi_{(v_\ell,\ell)}\colon 
    \mathcal{M}_{|U_\ell}\to \mathcal{N}_{|U_\ell}$
    be the morphism
    in Theorem \ref{Analytic lift}.
    Set $\tau\coloneqq \Xi_{k|V_k}\circ \Xi_{k-1|V_k}^{-1}$.
    By Theorem \ref{vert},
    there exist morphisms 
    $\tau_{\ell}\colon \mathcal{N}_{|U_\ell}\to \mathcal{N}_{|U_\ell}$,
    $\ell=k,k-1$ 
    such that 
    $\tau=\tau_{k|V_k}^{-1}\circ \tau_{k-1|V_k}$
    and that $\tau_\ell$ is asymptotic to the identity on $U_\ell$. 
    We set 
    ${}^\tau\Xi_\ell\coloneqq \tau_\ell\circ \Xi_\ell$. 
    On $V_k$, the following 
    equality holds:
    \begin{align*}
        {}^\tau\Xi_k
        &=\tau_k\circ \Xi_k=\tau_k\circ \tau\circ \Xi_{k-1}\\
        &=\tau_k\circ \tau_k^{-1}\circ\tau_{k-1}\circ \Xi_{k-1}
        ={}^\tau\Xi_{k-1}.
    \end{align*}
Hence we obtain a glued morphism ${}^\tau\Xi\colon \mathcal{M}_{|U_k\cup U_{k-1}}\to \mathcal{N}_{|U_{k}\cup U_{k-1}}$,
which defines a desired 
isomorphism
$\Xi_{U}$ on a neighborhood 
$U$ of $x=\i^k$.
\end{proof}

\subsubsection{}\label{PreProof}
We prepare some notation
to prove Theorem \ref{vert}. 
We consider the 
decomposition 
$\scr{N}\otimes \C\conv{s^{{-1/p}}}
=\bigoplus_{\frak{j}=1}^{\frak{m}} \scr{E}^{\frak{a_j}}\otimes \scr{R}_{A_{\frak{j}}} $.
Hence, we may take 
a basis $(\varepsilon_{{j}}^\infty)_{{j}}$
over $\C\conv{s^{-1/p}}$
that is compatible with this decomposition.
We have
$\varepsilon_j^\infty=\sum_i 
p_{ij}^\infty(s)\epsilon_i^\infty$
for some $p^\infty_{ij}(s)\in \C\conv{s^{-1/p}}$. 
Let $P=(p_{ij}^\infty)_{ij}$
be the corresponding invertible matrix.

For each $k\in \Z$,
set $W_k=U_k\cup U_{k-1}$.
Fix a branch of 
$s^{1/p}$. 
Take a sufficiently large compact subset $K$
in the definition
of quadrants $U_\ell$. 
Then $p_{ij}^\infty(s)$
is represented by 
a holomorphic function $p_{ij}$
on $W_k$. 
Then, 
$\varepsilon_i=\sum_{j}p_{ij}\epsilon_j$
defines 
a global frame on $\mathcal{N}_{|W_k}$. 
In this way,
we obtain a decomposition 
$\mathcal{N}_{|W_{k}}=\bigoplus_{\frak{j}=1}^{\frak{m}} \mathcal{N}_{\frak{j}}$
with $\Psi^\gr=\bigoplus_{\frak{j}} \Psi^\gr_{\frak{j}}$.
Let $r_{\frak{j}}$ denote the rank of 
$\cal{N}_\frak{j}$.

For $Y=U_\ell$ $(\ell\in\{k,k-1\})$ or $Y=V_k$,
a morphism 
$\sigma \colon\mathcal{N}_{|Y}\to \mathcal{N}_{|Y}$
is asymptotic to the identity
on $Y$
if and only if 
the representation matrix 
with respect to 
the basis $(\varepsilon_{\frak{j}})_{\frak{j}}$
is asymptotic to the 
identity matrix. 
Hence, in the proof 
of Theorem \ref{vert},
we mainly consider 
such a matrix presentation.

The morphism
$\sigma\colon \mathcal{N}_{|Y}\to \mathcal{N}_{|Y}$
compatible with $\Psi^\gr$
is decomposed into 
$\sigma_{\frak{i,j}}\colon \mathcal{N}_{\frak{j}}\to \mathcal{N}_{\frak{i}}$
($\frak{i,j}=1,\dots,\frak{m}$). 
The matrix presentation of
$\sigma_{\frak{i,j}}$ with respect
to $(\varepsilon_{\frak{j}}
)_{\frak{j}}$
is of the form 
$\exp(\frak{a}_{\frak{j}}(u)-\frak{a}_{\frak{i}}(s))s^{-A_{\frak{i}}}T_{\frak{ij}}(u)s^{A_{\frak{j}}}$
where $T_{\frak{ij}}(u)$
is a $r_{\frak{i}} \times r_{\frak{j}}$-matrix
whose entries are periodic functions 
(hence functions of $u$)
on $Y$. 
Main efforts 
in the proof 
will be devoted to the 
decomposition 
of the block 
matrix $\begin{bmatrix}T_{\frak{ij}}(u)
\end{bmatrix}_{\frak{i,j}}$
on $Y=V_k$
into the product of
matrices defined 
on $Y=U_\ell$ $(\ell=k,k-1)$. 

We fix the notation 
for $\frak{a}_{\frak{j}}$:

\begin{align}\label{a-}
    \frak{a}_{\frak{j}}
    =p^{-1}q_{\frak{j}}s
    \log s+\sum_{0<{\mu}<p}
    c_{\frak{j},\frak{\mu}} 
    s^{\mu/p}+c_{\frak{j}}s,
\end{align}
where $p,q_{\frak{j}}\in\mathbb{Z}$, $p>0$,
$c_{\frak{j},\mu}\in\mathbb{C}$
for $\mu=1,\dots,p-1$,
and $c_{\frak{j}}\in\C$.
Set $\lambda_{\frak{j}}=p^{-1}q_{\frak{j}}$.

\subsubsection{}\label{UpDown}
We shall prove Theorem \ref{vert}
when $k\equiv 0\mod 2$. 
We may assume that 
$k=0$ without loss of 
generality. 
Set $U_+\coloneqq U_0$,
$U_-\coloneqq U_{-1}$, 
$V\coloneqq V_0=U_+\cap U_-$, and
$W\coloneqq W_0=U_+\cup U_-$. 
Fix a branch of 
$\log s$ on $W$
so that $\log s\in \R$
for $s\in \R_{>0}\cap W$. 
We take an order of 
$(\frak{a}_{\frak{j}}\mid
\frak{j}=1,\dots,\frak{m})$
so that 
$\frak{i}<\frak{j}$ implies
one of the following conditions:
    \begin{itemize}
        \item[(a)] $q_{\frak{i}}
        >q_{\frak{j}}$.
        \item[(b)] $q_{\frak{i}}
        =q_{\frak{j}}$ and 
        $\mathrm{Re}(c_{\frak{i}})>
        \mathrm{Re}(c_{\frak{j}})$.
        \item[(c)] 
        $q_{\frak{i}}
        =q_{\frak{j}}$,
        $\mathrm{Re}(c_{\frak{i}})
        =\mathrm{Re}(c_{\frak{j}})$
        and the following condition holds:
        If there exists $\mu$
        such that $\mathrm{Re}(
        c_{\mu,{\frak{i}}}-c_{\mu,\frak{j}}) \neq 0$,
        then
        for $\mu_0=\max\{\mu\mid\mathrm{Re(}
        c_{\mu,\frak{i}}-c_{\mu,\frak{j}}) \neq 0\}$,
        the inequality
        $\mathrm{Re}(c_{\mu_0,\frak{i}}-c_{\mu_0,\frak{j}})>0$ holds. 
    \end{itemize}

Let $\tau\colon \mathcal{N}_{|V}\to \mathcal{N}_{|V}$ be 
an automorphism 
of $\mathcal{N}_{|V}$ 
compatible with 
$\Psi^\gr$ and 
asymptotic to the identity on $V$.
We would like to find 
morphisms 
$\tau_0
=\tau_+
\colon \mathcal{N}_{|U_+}
\to \mathcal{N}_{|U_+}$ 
and $\tau_{-1}
=\tau_-\colon 
\mathcal{N}_{|U_-}
\to \mathcal{N}_{|U_-}$
satisfying the properties 
in Theorem \ref{vert}.

Let $\mathcal{N}_{|W}=
\bigoplus_{\frak{j}=1}^{\frak{m}}
\mathcal{N}_{\frak{j}}$
be the decomposition
corresponding to $\mathscr{N}\otimes\C\conv{s^{-1/p}}$
described in \S \ref{PreProof}.
Let $\tau_{\frak{ij}}\colon\mathcal{N}_{\frak{j}|V}\to \mathcal{N}_{\frak{i}|V}$ be the 
block component of $\tau$ with respect to the 
decomposition $\mathcal{N}_{|W}=\bigoplus_{\frak{j}=1}^{\frak{m}}\mathcal{N}_{\frak{j}}$.
Symbols $\tau_{\pm,\frak{i,j}}$ 
and $\tau_{\pm,\frak{i,j}}^{-1}$
also denote 
the block components of $\tau_\pm$ 
and $\tau_{\pm}^{-1}$ in a similar way. 
\begin{lemma}
    We have $\tau_{\frak{ij}}=0$ 
    for $\frak{i}>\frak{j}$
    and $\tau_{\frak{i}}=\id_{\mathcal{N}_{\frak{i}|V}}$
    for $\frak{i}=1,\dots,\frak{m}$.
\end{lemma}
\begin{proof}
    As discussed in \S \ref{PreProof},
    the representation matrix $\tau_{\frak{ij}}$
    with respect to the bases 
    $(\varepsilon_{j})$
    is of the form 
    $\exp(\frak{a}_{\frak{j}}(u)-\frak{a}_{\frak{i}}(s))s^{-A_{\frak{i}}}T_{\frak{ij}}(u)s^{A_{\frak{j}}}$,
    where $T_{\mathfrak{ij}}(u)$ 
    is a matrix-valued holomorphic 
    function with variable $u$. 
    By the ordering, 
    $\exp(\frak{a}_{\frak{j}}-\frak{a}_{\frak{i}})$
    when $s\in V$, $s\to \infty$
    rapidly grows for 
    $\frak{i>j}$. 
    Since $\tau$ is asymptotic to the identity, 
    we should have
    $T_{\frak{ij}}=0$ for $\frak{i>j}$.
    We also have 
    $\tau_{\frak{i,i}}
    =\id_{\mathcal{N}_{\frak{i}}}$
    since $T_{\frak{ii}}(u)$
    is a function of $u$
    and should be asymptotic to the identity 
    on $V$. 
\end{proof}
Hence, $\tau$
is upper triangular with identity 
diagonal blocks
with respect to the decomposition.
A standard argument on 
the decomposition 
of upper triangular matrices reduces to demonstrating the following:

\begin{lemma}\label{additive}
    Take $\frak{i,j}\in \{1,\dots,\frak{m}\}$
    with $\frak{i<j}$. 
    Assume that 
    $h(u)\exp(\frak{a}_{\frak{i}}-\frak{a}_{\frak{j}})$
    is asymptotic to zero
    on $V$, where
 $h(u)$ denotes a holomorphic function 
    in $u=\exp(2\pi \i s)$ on an annulus $\{u\mid e^{a}<|u|<e^b\}$
    for some real numbers $a<b$. 
    Then 
    the function $h$ can be written as a sum 
    \[h(u)=h_+(u)+h_-(u)\]
    with the following properties:
    \begin{itemize}
        \item $h_+(u)\exp(\frak{a}_{\frak{j}}-\frak{a}_{\frak{i}})$
        is asymptotic to zero
        on $U_+$. 
        \item $h_-(u)\exp(\frak{a}_{\frak{j}}-\frak{a}_{\frak{i}})$
        is asymptotic to zero 
        on $U_{-}(\epsilon)$
        for every $\epsilon>0$
        and moderate growth in $u^{-1}$
        on any vertical strip. 
    \end{itemize}
    
\end{lemma}
\begin{proof}
Let $h(u)=\sum_{n\in \mathbb{Z}} h_n u^n$
be the Laurent expansion of 
$h$. 
We shall take a real number 
$A$ such that 
$h_+(u)=\sum_{n>A}h_nu^n$
and 
$h_-(u)=\sum_{n\leq A}h_nu^n$
satisfy
the desired condition. 
We
consider
the asymptotic behavior of
$u^n\exp(\frak{a}_{\frak{j}}(s)-\frak{a}_{\frak{i}}(s))$. 
For case 
(a) $q_{\frak{i}}>q_{\frak{j}}$
, set $A=(4p)^{-1}(q_{\frak{i}}-q_{\frak{j}})+(2\pi)^{-1}\mathrm{Im}(c_{\frak{j}}-c_{\frak{i}})$;
for cases 
$(b)$,
set $A=(2\pi)^{-1}\mathrm{Im}(c_{\frak{j}}-c_{\frak{i}})$;
for case 
$(c)$,
set $A=(2\pi)^{-1}\mathrm{Im}(c_{\frak{j}}-c_{\frak{i}})$ or $A=(2\pi)^{-1}\mathrm{Im}(c_{\frak{j}}-c_{\frak{i}})-\delta$
for sufficiently small $\delta>0$.
Then, in each case, 
we obtain the desired result
(see pp.~136--137 of \cite{Galois}).
\end{proof}

We shall finish the proof 
by recalling the standard 
argument in this situation. 
For the given morphism
$\tau_{\frak{ij}}$ represented by the matrix $\exp(\frak{a}_{\frak{j}}(u)-\frak{a}_{\frak{i}}(s))s^{-A_{\frak{i}}}T_{\frak{ij}}(u)s^{A_{\frak{j}}}$ for $\frak{i<j}$, 
it reduces to consider the equation 
\begin{align*}
    \tau_{\frak{i,j}}-\sum_{\frak{i<k<j}}\tau^{-1}_{+,\frak{i,k}}\tau_{-,\frak{k,j}}
    =\tau^{-1}_{+,\frak{i,j}}+\tau_{-,\frak{i,j}}
\end{align*}
by induction on $d_{\frak{ij}}=\frak{j}-\frak{i}\geq 0$.
Then,
we may apply Lemma \ref{additive}
to obtain $\tau_{+,\frak{i,j}}$ and $\tau_{-,\frak{i,j}}$
in each induction step.
Hence, we obtain $\tau_+$ and $\tau_-$
with the desired property. 
\subsubsection{}
We shall give a proof of Theorem \ref{vert} when $k\equiv 1\mod 2$.
We may assume that $k=1$ without loss 
of generality. 
We set $U_R=U_0$,
$U_L=U_1$, and $V=U_R\cap U_L$
and $W=U_R\cup U_L$ in the following. 
Fix a branch of $\log s$
on $W$ so that $\mathrm{Im}(\log s)=\pi/2 $
for $s\in i\R\cap W$. 

Let $\tau\colon \mathcal{N}_{|V}\to \mathcal{N}_{|V}$ be 
an automorphism 
of $\mathcal{N}$ compatible with 
$\Psi^\gr$ and asymptotic 
to the identity on $V$.
We would like to find 
morphisms $\tau_{1}=\tau_L\colon \mathcal{N}_{|U_L}\to \mathcal{N}_{|U_L}$ 
and $\tau_{0}=\tau_R\colon \mathcal{N}_{|U_R}
\to \mathcal{N}_{|U_R}$
satisfying the properties in Theorem \ref{vert}.
We consider the decomposition $\mathcal{N}_{|W}=
\bigoplus_{\frak{j}=1}^{\frak{m}}
\mathcal{N}_{\frak{j}}$
corresponding to $\mathscr{N}\otimes\C\conv{s^{-1/p}}$
described in \S \ref{PreProof}.
We use the notations 
$\tau_{\frak{i,j}}$, 
$\tau_{L,\frak{i,j}}$, 
$\tau_{R,\frak{i,j}}$ as in the previous sections.

Let $\frak{a}_{\frak{j}}$ be as in the previous section with the expression \eqref{a-}. 
We take an order of 
$(\frak{a}_{\frak{j}}\mid
\frak{j}=1,\dots,\frak{m})$
so that 
$\frak{i<j}$ implies
one of the following conditions:
\begin{itemize}
        \item[(i)] $q_{\frak{i}}<q_{\frak{j}}$. 
        \item[(ii)] $q_{\frak{i}}=q_{\frak{j}}$ and 
        $\mathrm{Re}(c_{\frak{i}})<\mathrm{Re}(c_{\frak{j}})$.
        \item[(iii)] 
        $q_{\frak{i}}=q_{\frak{j}}$,
        $\mathrm{Re}(c_{\frak{i}})=\mathrm{Re}(c_{\frak{j}})$
        and the following condition holds:
        If there exists $\mu$
        such that $\mathrm{Re}[\exp(\mu\pi\i/2p)(
        c_{\mu,{\frak{i}}}- c_{\mu,\frak{j}})] \neq 0$,
        then the inequality
        \[\mathrm{Re}[\exp(\mu_0\theta\i/p)(
        c_{\mu_0,{\frak{i}}}- c_{\mu_0,\frak{j}})]>0\]
        holds
        for $\pi/2\leq \theta\leq \pi$ and
        $\mu_0=\max\{\mu\mid\mathrm{Re}[\exp(\mu\pi\i/2p)(
        c_{\mu,{\frak{i}}}- c_{\mu,\frak{j}})] \neq 0\}$. 
    \end{itemize}
Under this ordering,
we obtain the following:
\begin{lemma}\label{LUlem}
    Consider the function
    $f_{n,\frak{i,j}}(s)=u^n(s)\exp(\frak{a}_{\frak{j}}(s)-\frak{a}_{\frak{i}}(s))$
    on $W$. 
    Assume that 
    $f_{n,\frak{i,j}}(s)$ is asymptotic to zero on $V$.
    Then, we have 
    \begin{align}\label{nnn}
     n\geq 4^{-1}(\lambda_{\frak{j}}-\lambda_{\frak{i}})+(2\pi)^{-1}\mathrm{Im}(c_{\frak{i}}-c_{\frak{j}})   
    \end{align}
    and the following:
    \begin{itemize}
        \item If $\frak{i<j}$,
        then $f_{n,\frak{i,j}}$ decays rapidly on 
        $U_L$.
        \item If $\frak{i>j}$,
        then $f_{n,\frak{i,j}}$
        decays rapidly on $U_{R}(\epsilon)$
        for every $\epsilon>0$. 
    \end{itemize}
\end{lemma}
\begin{proof}
    Using the expression \eqref{a-},
    we have
    \begin{align*}
        \log |f_{n,\frak{i,j}}(s)|=&
        -n 2\pi\mathrm{Im}(s)+
        (\lambda_\frak{j}-\lambda_{\frak{i}})
        \{\mathrm{Re}(s)\log |s|-\arg(s)\mathrm{Im}(s)\}\\
       &+\mathrm{Re}\left[g_{\frak{i,j}}(s)\right]+\mathrm{Re}(c_{\frak{j}}-c_{\frak{i}})\mathrm{Re}(s)-\mathrm{Im}(c_{\frak{j}}-c_{\frak{i}})
       \mathrm{Im}(s)
    \end{align*}
    where we put $g_{\frak{i,j}}(s)=\sum_{0<\mu<p}(c_{\frak{j},\mu}-c_{\frak{i},\mu})s^{\mu/p}$.
    When $|s|\to \infty$, $s\in V$,
    $\mathrm{Re}(s)$ is bounded and $\mathrm{Im}(s)$
    goes to $+\infty$ $(\arg(s)\to \pi/2)$. Hence, if $f_{n,\frak{i,j}}$ decays rapidly on $V$, then \eqref{nnn} must hold. 
    When equality holds, the rapid 
    decay of
    $f_{n,\frak{i,j}}$ 
     on $V$ depends on the behavior of  $\mathrm{Re}[g_{\frak{i,j}}(s)]$. 
    If there exists $ \mu$ such that 
    $\mathrm{Re}[\exp(\mu\pi\i/2p)(
        c_{\mu,{\frak{i}}}- c_{\mu,\frak{j}})] \neq 0$, 
    $ \mathrm{Re}[\exp(\mu_0\pi\i/2p)(
        c_{\mu_0,{\frak{j}}}- c_{\mu_0,\frak{i}})]<0$
        implies that $ f_{n,\frak{i,j}}$ decays rapidly on $V$ when equality for \eqref{nnn} holds.

    Consider the case $\frak{i}<\frak{j}$
    and (i) $q_{\frak{i}}<q_{\frak{j}}$ holds. 
    In this case, the leading term is $(\lambda_\frak{j}-\lambda_{\frak{i}})\mathrm{Re}(s)\log|s|$ and hence  $f_{n,\frak{i,j}}$ decays rapidly on $U_L$. 
    
     Consider the case $\frak{i}<\frak{j}$
    and the condition (ii) above holds. Then the leading term is $\mathrm{Re}(c_{\frak{j}}-c_{\frak{i}})\mathrm{Re}(s)$,
    which implies the desired asymptotics.
    
      In the case $\frak{i}<\frak{j}$
    and the condition (iii) above holds, the leading term $\mathrm{Re}[g_{\frak{i,j}}]$ has the desired asymptotic behavior
    when $\arg(s)=\pi$ or the equality in \eqref{nnn} holds. When $\pi/2 <\arg(s)<\pi-\epsilon $ ($\epsilon>0$) and equality 
    \eqref{nnn} does not hold, we see that $f_{n,\frak{i,j}}$ decays rapidly. 

    The other cases are proved similarly and therefore omitted. 
\end{proof}

    Let 
    $\exp
    (\frak{a}_{\frak{j}}(u)-\frak{a}_{\frak{i}}(s))
    s^{-A_{\frak{i}}}T_{\frak{ij}}
    (u)s^{A_{\frak{j}}}$
    denote the matrix representation 
    of $\tau_{\frak{i,j}}$. 
    By Lemma \ref{LUlem},
    the entries of $T_{\frak{ij}}(u)$
    are in $\C\conv{u}$.
    Since $\tau$ is asymptotic to the identity, 
    the determinant 
    of the matrix
    $\begin{bmatrix}
        \exp
    (\frak{a}_{\frak{j}}(u)-\frak{a}_{\frak{i}}(s))
    s^{-A_{\frak{i}}}T_{\frak{ij}}
    (u)s^{A_{\frak{j}}}
    \end{bmatrix}_{1\leq \frak{i,j}\leq r}$ 
    is non-zero for $r=1,\dots,m$.
    Hence the determinant of the matrix 
    $\begin{bmatrix}
        T_{\frak{ij}}
    (u)
    \end{bmatrix}_{1\leq \frak{i,j}\leq r}$
    is also non-zero. 

    Then we obtain a block matrix 
    version of the LU decomposition
    of $T=(T_{\frak{ij}}(u))$ over the field $\C\conv{u}$,
    namely
    \[
T=T_{R}T_{L}^{-1}
    \]
    where
    $T_R=(T_{R,\frak{ij}})$ and $T_{L}=(T_{L,\frak{ij}})$
    satisfy
    $T_{R,\frak{i,j}}=0$ for $\frak{i<j}$, 
     $T_{L,\frak{i,j}}=0$ for $\frak{i>j}$, and 
    $T_{L,\frak{ii}}=\mathrm{id}$ for $\frak{i}=1,\dots,\frak{m}$.

    Let
    $\tau_{L,{\frak{i,j}}}\colon\mathcal{N}_{\frak{j}|U_L}\to \mathcal{N}_{\frak{i}|U_L}$
    (resp. $\tau_{R,\frak{i,j}}$) be the morphism 
    represented by the matrix
    $\exp
    (\frak{a}_{\frak{j}}(u)-\frak{a}_{\frak{i}}(s))
    s^{-A_{\frak{i}}}T_{L, \frak{ij}}
    (u)s^{A_{\frak{j}}}$ 
    (resp. $\exp
    (\frak{a}_{\frak{j}}(u)-\frak{a}_{\frak{i}}(s))
    s^{-A_{\frak{i}}}T_{R, \frak{ij}}
    (u)s^{A_{\frak{j}}}$).
    Then, by Lemma \ref{LUlem},
    we obtain the desired morphisms
    $\tau_L$ and $\tau_R$. 
\qed

\subsection{Solution to the inhomogeneous 
equation}\label{Sol to inhom eq}
In this subsection,
we prove Theorem \ref{Claim}.
Let $\scr{M}$ be a difference module 
    over $\C\conv{s^{-1}}$. 
    Let $U\subsetneq  S^1$ be 
    an open interval. 
    Assume that 
    $(\mathscr{Q}^{\leqslant 0} \otimes \scr{M}_{S^1})_{|U}\simeq \mathscr{Q}^{\leqslant 0}_{|U} \otimes (\scr{E}^{\frak{a}}\otimes \scr{R}_{G})_{S^1|U}$
for 
$\frak{a}(s)= 
    qs\log (s^{1/p})+
    \sum_{i=1}^p c_i s^{i/p}$
    with $q\in \Z$,
    $c_1,\dots,c_p\in \C$,
and for $G=\gamma\id_{\C^r}+N$ with 
$\gamma\in\C$ and $N\in \End(\C^r)$ a nilpotent 
matrix.
We set $A(s)=\exp(\frak{a}(s+1)-\frak{a}(s))
(1+s^{-1})^{G}$.

Let $f$ be a local section of 
$\mathscr{Q}^{\leqslant 0} \otimes \scr{M}_{S^1}$ at $e^{\i\theta}\in U$.
We consider the three cases:
\begin{enumerate}
    \item $e^{\i\theta}\in S^1\setminus \{\pm \i^k\mid k\in \Z\}$,
    \item $e^{\i\theta}\in\{\i,-\i\}$
    \item $e^{\i\theta}\in \{\pm 1\}$.
\end{enumerate}
In each case, 
essential estimates on the constructed solution
$\Lambda_\theta({f})$ have already been 
carried out by Immink \cites{Immink,Immink2}.
Here, we only recall the constructions
in each case
referring to the relevant estimates in \cite{Immink}, and check that 
$\Lambda_\theta({f})$ is globally defined
in each case. 
\subsubsection{The case $(1)$}
For the first case, we may assume 
$0<\theta<\pi/2$ without loss of generality.  
We then consider three cases:
(i) $q=0$, (ii) $q>0$, and (iii) $q<0$.

The case (i) $q=0$ is the mild case and
essentially the same construction as in \cite{shamoto2022stokes} works (see also the case (iii)). 
For case (ii) $q>0$,
recall that $f$ is defined over an 
upper half-plane 
$\{s\mid \mathrm{Im}(s)>M\}$
for sufficiently large $M>0$.
Then we 
can take 
\begin{align}\label{Lambda I}
    \Lambda_\theta(f)=-f(s)-\sum_{n=0}^\infty
    A(s)\cdots A(s+n)f(s+n+1)
\end{align}
as a solution of $[\widetilde{\psi}-\id](\Lambda_\theta(f))=f$. 
The estimates for this case can be found in \cite{Immink}*{\S 9.4}. 
For case (iii) $q<0$,
we take an appropriate 
path $C(s)$
depending on $s\in \C$
and consider the integral
\begin{align}\label{Lambda}
    \Lambda (f)(s)\coloneqq
-f(s)+
    Y(s)
    \int_{{{C}} (s)}
    \frac{Y(\zeta)^{-1}f(\zeta)}
    {1-e^{ 2\pi\i (s-\zeta) }}d\zeta 
\end{align}
where 
$Y(s)=\exp(-\frak{a}(s))s^{-G}$ denotes the fundamental 
solution of the difference equation
corresponding to $\scr{E}^\frak{a}\otimes \scr{R}_G$.

In constructing the path $C(s)$,
we use the condition that 
$f$ is defined over $\scr{Q}$.
Replacing $f$ with 
$u^nf$ if necessary,
we may assume that 
$Y(\zeta)^{-1}f(\zeta)$ is 
``rapid decay along each vertical strip,"
i.e.,
for any $a< b$ and $N>0$,
there exists 
a constant $C>0$ and $R>0$
such that 
$\|Y(s)^{-1}f(s)\|<C|s|^{-N}$
for any $s$ with $\mathrm{Im}(s)>R$
and $a<\mathrm{Re}(s)<b$,
where $\|\cdot \|$ denotes a 
norm on $\scr{M}$
defined using a meromorphic
frame of $\scr{M}$. 
Fix a point 
$c_1\in \{s\mid \mathrm{Im}(s)>M\}$
and take $M'>\mathrm{Im}(c_1)$.
We choose a piecewise linear path 
$C(s)$
for any $s\in \{s\mid \mathrm{Im}(s)>M'\}$
as follows:
\[C(s;t)=
\begin{cases}
    (1-t)c_1+ts' &( 0\leq t\leq 1)\\
    s'+(t-1)\i&(1\leq t)
\end{cases}
\]
where $s'=s+1/2$.
Then, we can define 
$\Lambda(f)$ 
in \eqref{Lambda} using this path $C(s)$.
We obtain 
$[\widetilde{\psi}-\id](\Lambda(f))=f$
similarly as in the mild case
\cite{shamoto2022stokes}.
The relevant estimate,
namely
the reduction to the use of
\cite{Immink}*{Lemma 8.12},
can be found in \cite{Immink}*{\S 12}.

\subsubsection{The case $(2)$}
We may assume that 
$\theta=\pi/2$. 
Then the construction in case (1)(iii)
also works in this case for every $q\in \Z$.

\subsubsection{The case $(3)$}
We may assume that $\theta=0$. 
In this case, 
$f$ is defined on a region 
of the form 
$\{s\mid \mathrm{Re}(s)>A\}\cup
\{s\mid |\mathrm{Im}(s)|>B\}$
for some positive $A,B$.
For the case (ii) $q>0$,
the construction \eqref{Lambda I}
also works
in this case. 
For the case (iii) $q<0$,
as in the case $(1)$, we may assume that 
$Y(s)^{-1}f(s)$ decays rapidly on 
the vertical strip 
$\{s\in \C\mid a<\mathrm{Re}(s)<b\}$
when 
$\mathrm{Im}(s)\to
+\infty$ for any $a<b$.
Set $c_2=A+1/2$.
For $s\in \{s\in \C\mid \mathrm{Re}(s)>A+1\}$,
we consider a piecewise linear path 
$C(s)=C(s;-)\colon \mathbb{R}_{\geq 0}\to \C$
defined as follows:
\[C(s;t)=\begin{cases}
    (1-t)c_2+ts' &( 0\leq t\leq 1)\\
    s'+(t-1)\i&(1\leq t)
\end{cases}\]
where $s'=s+1/2$.
Then, we can define 
$\Lambda_\theta(f)$ 
in \eqref{Lambda} using this path $C(s)$
for $\mathrm{Re}(s)>A+1$. 
We obtain 
$[\widetilde{\psi}-\id](\Lambda_\theta(f))=f$
similarly.
Then $\Lambda_\theta(f)(s)$ can be
analytically continued 
to the whole region $\{s\mid \mathrm{Re}(s)>A+1\}\cup
\{s\mid |\mathrm{Im}(s)|>B\}$
by this difference equation.

\bibliographystyle{alpha}
\bibliography{Difference}
\end{document}